\documentclass[10pt,reqno,a4paper]{amsart}

\usepackage[pdftex]{graphicx}
\usepackage{amsmath}
\usepackage{amssymb,amsthm,hyperref,caption}
\usepackage{color}
\definecolor{meinBlau}{rgb}{0.2,0.2,0.9} 
\definecolor{blau}{rgb}{0,0,0.75} 
\definecolor{rot}{rgb}{0.74,0,0} 
\hypersetup{colorlinks,linkcolor=blau,citecolor=blue,urlcolor=meinBlau}
\usepackage{times}
\usepackage{enumitem}

\allowdisplaybreaks

\newtheorem{theorem}{Theorem}
\newtheorem{lem}[theorem]{Lemma}
\newtheorem{coroll}{Corollary}
\newtheorem{prop}{Proposition}

\theoremstyle{definition}

\newtheorem{remark}{Remark}

\def\P{{\mathbb {P}}}
\def\E{{\mathbb {E}}}
\def\Var{{\mathbb {V}}}

\newcommand{\fallfak}[2]{\ensuremath{#1^{\underline{#2}}}}
\newcommand{\auffak}[2]{\ensuremath{#1^{\overline{#2}}}}

\newcommand{\N}{\ensuremath{\mathbb{N}}}

\newcommand{\Stir}[2]{\genfrac{ \{ }{ \} }{0pt}{}{#1}{#2}}

\DeclareMathOperator{\law}{\overset{\mathcal{L}}{=}}
\DeclareMathOperator{\claw}{\overset{\mathcal{L}}{\rightarrow}}

\DeclareMathOperator{\Unif}{Unif}
\DeclareMathOperator{\Geom}{Geom}
\DeclareMathOperator{\Exp}{Exp}
\DeclareMathOperator{\LinExp}{LinExp}
\DeclareMathOperator{\Rayleigh}{Rayleigh}
\DeclareMathOperator{\Bernoulli}{Bernoulli}
\DeclareMathOperator{\BetaDist}{Beta}
\DeclareMathOperator{\Cov}{Cov}
\DeclareMathOperator{\Corr}{Corr}

\begin{document}

\author[M.~Kuba]{Markus Kuba}
\address{Markus Kuba\\
Department Applied Mathematics and Physics\\
University of Applied Sciences - Technikum Wien\\
H\"ochst\"adtplatz 5, 1200 Wien} %
\email{kuba@technikum-wien.at}

\author[A.~Panholzer]{Alois Panholzer}
\address{Alois Panholzer\\
Institut f{\"u}r Diskrete Mathematik und Geometrie\\
Technische Universit\"at Wien\\
Wiedner Hauptstr. 8-10/104\\
1040 Wien, Austria} \email{Alois.Panholzer@tuwien.ac.at}

\title[On Card guessing with two types of cards]{On Card guessing with two types of cards}


\begin{abstract}
We consider a card guessing strategy for a stack of cards with two different types of cards, say $m_1$ cards of type red (heart or diamond) and $m_2$ cards of type black (clubs or spades). Given a deck of $M=m_1+m_2$ cards, we propose a refined counting of the number of correct color guesses, when the guesser is provided with complete information, in other words, when the numbers $m_1$ and $m_2$ and the color of each drawn card are known. We decompose the correct guessed cards into three different types by taking into account the probability of making a correct guess, and provide joint distributional results for the underlying random variables as well as joint limit laws.
\end{abstract}

\maketitle

\section{Introduction}

Card guessing games have been considered in the literature in many articles~\cite{DiaconisGraham1981,HeOttolini2021,KnoPro2001,KuPanPro2009,Leva1988,OttoliniSteiner2022,Read1962,Sulanke,Zagier1990}. As the basic setting a randomized deck of $M$ cards is considered. 
A person is asked to guess the card on top of the deck. Afterwards, the top card is revealed to the guesser and then discarded. 
This process is continued until no more cards are left. If the person guessing the cards knows in advance the composition of the
deck before randomization, say the number of hearts, diamonds, clubs and spades, respectively, one is interested in the number of correct guesses, its distribution, expectation and limit laws, under the assumption that the guesser uses this information to maximize the chance of a correct guess in each step. This is sometimes generalized to $n\ge 2$ types of different cards~\cite{DiaconisGraham1981,HeOttolini2021,OttoliniSteiner2022}, but so far only in terms of the expected value.
There are also variations of the card guessing procedure with only partial information revealed, we refer to the works of Diaconis et al.~\cite{Diaconis1978,Diaconis2022a,Diaconis2022b}, and also to~\cite{BlackwellHodges1957,Efron1971} for applications in clinical trials.

\smallskip

We consider the setting where two different kinds of cards are considered, say red (heart or diamond) and black (clubs or spades).
Their numbers are given by non-negative integers $m_1$, $m_2$, with $M=m_1+m_2$. We are interested in the random variable $C_{m_1,m_2}$, counting the number of correct guesses. This quantity has been studied in~\cite{KnoPro2001,KuPanPro2009,Leva1988,Read1962,Sulanke,Zagier1990}:
Sulanke~\cite{Sulanke} obtained the probability mass function of $C_{m_1,m_2}$, which was later rederived by Knopfmacher and Prodinger~\cite{KnoPro2001}. Zagier~\cite{Zagier1990} considered the distribution and the expectation of $C_{m,m}$. The epected value is also covered by the before mentioned general results for $n\ge 2$ different types. A more general result for the expected value
for arbitrary $m_1,m_2$ was obtained in~\cite{KnoPro2001}. We also mention the work~\cite{KT2021}, 
which rederives results for $C_{m_1,m_2}$, but seems to be unaware of the previous works. Moreover, limit laws for $C_{m_1,m_2}$ have been derived in~\cite{KuPanPro2009}, however without stating proofs, leading to phase transitions according to the different growths of $m_1$ and $m_2$ in the general case $\max\{m_1,m_2\}\to\infty$.

\bigskip

In the following we gain more insight into the number of correct guesses, its structure, as well as its limit laws. 
We decompose the correct guessed cards $C_{m_1,m_2}$ into three different types according to the probability $p\in[0.5,1]$ of a correct guess. Given $m_1$ red and $m_2$ black cards, where without loss of generality we restrict ourselves to $0\le m_2\le m_1$, we are interested in the number of boundary cases, say $p=0.5$, or $p=1$, or the cases in between, $0.5 < p <1$. In more detail, this means that we consider draws, where the guesser knows for sure he will be correct, $p=1$, 
second, draws, where the guesser has a pure luck fifty-fifty chance, and third, draws, where the guesser has a chance for a correct guess strictly between fifty and one hundred percent, $0.5<p<1$. 

\smallskip 

Let $T_{m_1,m_2}$ denote the number of certified correct guesses, $L_{m_1,m_2}$ denote the number of more likely guesses,
and $P_{m_1,m_2}$ the number of pure luck guesses. These random variables refine the total number of correct guesses, as we have the identity
\[
C_{m_1,m_2}=T_{m_1,m_2}+L_{m_1,m_2}+P_{m_1,m_2}.
\]
Moreover, the three random variables will give a detailed insight in how big the advantage of an educated guesser 
is compared to just random guessing. We note in advance that intuition dictates for $m_1 \gg m_2$, the number of trivially correct guesses $T_{m_1,m_2}$ should dominate the asymptotics of $C_{m_1,m_2}$, whereas for $m_1 \sim m_2$, of the same order, we should observe an 
interplay between more likely guesses $L_{m_1,m_2}$ and pure luck $P_{m_1,m_2}$.

\smallskip

Interestingly, it will turn out that the random variable $T_{m_1,m_2}$ links the card guessing game directly to so-called
diminishing urn models~\cite{FlaDuPuy2006,HKP2007,KubPanClass}: one interprets the card guessing game as an urn containing two types of balls, 
say red and black. At random, balls are drawn from the urn, their color inspected and then removed. This is well known as a
sampling without replacement urn: $\left(\begin{smallmatrix} -1 & 0\\0 &-1 \end{smallmatrix}\right)$.
One is interested in the number of balls, when one type is emptied. Using a geometric interpretation via lattice paths, cf.\ Section~\ref{sec:Geometric_interpretation}, this corresponds to an absorption at the coordinate axes~\cite{HKP2007}, similar to $T_{m_1,m_2}$. 

\smallskip

We propose a generating function approach, based on Prodinger and Knopfmacher's analysis~\cite{KnoPro2001}, 
to determine the distribution and limit laws of the aforehand stated random variables. 
This allows to obtain a detailed insight into the nature of the limit laws~\cite{KuPanPro2009}, 
as well as the precise relation of the card guessing game to so-called diminishing urn models. 

\smallskip

As a remark concerning notation used throughout this work, we always write $X \law Y$ to express equality in distribution of two random variables (r.v.) $X$ and $Y$, and $X_{n} \claw X$ for the weak convergence (i.e., convergence in distribution) of a sequence
of random variables $X_{n}$ to a r.v.\ $X$. Furthermore we use $\fallfak{x}{s}:=x(x-1)\dots(x-(s-1))$ for the falling factorials, and $\auffak{x}{s}:=x(x+1)\dots(x+s-1)$ for the rising factorials, $s\in\N_0$. Moreover, $f_{n} \ll g_{n}$ denotes that a sequence $f_{n}$ is asymptotically smaller than a sequence $g_{n}$, i.e., $f_{n} = o(g_{n})$, $n \to \infty$.

\section{Distributional analysis}
\subsection{Decomposition of pure luck}
In order to analyze the number of pure luck guesses, we actually turn to another random variable, 
$W_{m_1,m_2}$, counting the number of times during the card guessing process when the number of red and black cards are equal and non-zero. This random variable allows to obtain a decomposition of pure luck guesses.

\begin{lem}[Decomposition of pure luck guesses]\label{lem:P-W-Relation}
Let $B(n,p)$ denote a Binomial distribution with success probability $p$ and $n$ trials. The random variable $P_{m_1,m_2}$, counting the number of pure luck guesses, 
is distributed as a Binomial distribution with a random number of trials:
\[
P_{m_1,m_2} \law B\big(W_{m_1,m_2},\textstyle{\frac{1}{2}}\big).
\]
\end{lem}  

\begin{proof}
By definition of $P_{m_1,m_2}$, each individual such guess happens when we have a success probability of $p=0.5$, thus, when both types of cards are equally many. At any time reaching a composition $(j,j) \neq (0,0)$ (for the pair consisting of the number of cards of each color), we thus can toss a fair coin, independent of what has happened before. Since $W_{m_1,m_2}$ counts the (random) number of times we reach such a state $(j,j) \neq (0,0)$ with equal numbers, the result follows.
\end{proof}

\subsection{Geometric interpretation via lattice paths and decompositions}\label{sec:Geometric_interpretation}
\begin{lem}\label{lem:L-T-C-Relation}
The random variables $L_{m_1,m_2}$ and $T_{m_1,m_2}$, counting the more likely guesses and the trivial guesses, satisfy
\begin{equation}\label{eqn:L-T-Relation}
  L_{m_1,m_2} + T_{m_1,m_2} = m_1.
\end{equation}
Consequently, the total number of correct guesses satisfies
\begin{equation}
  C_{m_1,m_2} = m_{1} + P_{m_1,m_2} \law m_{1} + B\big(W_{m_1,m_2},\textstyle{\frac{1}{2}}\big).
\end{equation}
\end{lem}
\begin{remark}
We remark that the latter relation gives an explanation to the occurrence of a shift by $m_{1}$, appearing in \cite{KuPanPro2009} when stating the limiting distribution behaviour of $C_{m_1,m_2}$.
\end{remark}
\begin{proof}
For a geometric interpretation of the card guessing process, we think of the numbers $m_1$ and $m_2$ as the $(x,y)$-coordinates of a particle in the wedge $x\ge y\ge 0$. 
The evolution of the process is described via weighted lattice paths from $(m_1,m_2)$ to the origin with step sets ``left'' $(-1,0)$ and ``down'' $(0,-1)$. 

Namely, when $m_1 > m_2 \ge 0$, a left step $(m_1,m_2) \to (m_1-1,m_2)$ and a down step $(m_1,m_2) \to (m_1,m_2-1)$, resp., reflect the draw of a card of the first or second type, resp., and appear with probabilities $\frac{m_1}{m_1+m_2}$ and $\frac{m_2}{m_1+m_2}$, respectively. Consequently, a left step corresponds to a correct guess, either, if $m_2=0$, a certified correct guess increasing the counter of $T_{m_1,m_2}$, or, if $m_2 \ge 1$, a more likely correct guess increasing the counter of $L_{m_1,m_2}$. In our generating functions approach, this results in respective weights of the steps. Note that apart from introducing certain weighted left steps, this description simply corresponds to the aforehand mentioned interpretation of diminishing urns in terms of lattice paths.

However, when $m_1 = m_2 \ge 1$, a down step $(m_1,m_1) \to (m_1,m_1-1)$ occurs with probability $1$, which actually combines the draw of a card of the second type and the draw of a card of the first type, both appearing with probability $\frac{1}{2}$, where in the latter case we may think of exchanging the r\^{o}les of the first and second color, such that the amount of cards of the second type never exceeds that one of the first type.
A down step starting at the diagonal increases the counter of $W_{m_1,m_2}$ and, in the generating functions approach, results in a respective weight. As pointed out in Lemma~\ref{lem:P-W-Relation}, each down step from the diagonal corresponds to a Bernoulli trial, which yields with probability $\frac{1}{2}$ a correct pure luck guess.

\medskip

It follows from this geometric interpretation that each step to the left contributes either one to the number of certified correct guesses, which happens exactly if this step is on the $x$-axis, or one to the number of more likely correct guesses, if it is above the $x$-axis, whereas down-steps do not contribute to them. This implies the stated relation between the r.v.\ $L_{m_1,m_2}$ and $T_{m_1,m_2}$.
Moreover, together with Lemma~\ref{lem:P-W-Relation}, yields the stated representation of the number of correct guesses.
\end{proof}

\subsection{Generating function approach}
We define the multivariate generating function\newline
$\varphi_{m_1,m_2} := \varphi_{m_1,m_2}(u_1,u_2,w)$ as the multivariate probability generating function of\newline
$L_{m_1,m_2}$, $T_{m_1,m_2}$ as well as $W_{m_1,m_2}$:
\[
\varphi_{m_1,m_2}(u_1,u_2,w)=\E\big(u_1^{L_{m_1,m_2}}u_2^{T_{m_1,m_2}}w^{W_{m_1,m_2}}\big).
\]
By distinguishing the cases according to the first card drawn, we obtain the recurrence relation
\[
\varphi_{m_1,m_2}=u_1\frac{m_1}{m_1+m_2}\varphi_{m_1-1,m_2}+\frac{m_2}{m_1+m_2}\varphi_{m_1,m_2-1},
\]
with $M=m_1+m_2$, for $m_1>m_2>0$, and with initial values $\varphi_{m,0}=u_2^m$, $m\ge 0$. Moreover, 
\[
\varphi_{m,m}=w\varphi_{m,m-1},\quad m\ge 1.
\]
In order to simplify the analysis, we consider the quantity 
\[
\Phi_{m_1,m_2}(u_1,u_2,w)=\binom{m_1+m_2}{m_1}\varphi_{m_1,m_2}(u_1,u_2,w),
\]
leading to the simplified recurrence relations
\begin{equation}
\Phi_{m_1,m_2}=u_1\Phi_{m_1-1,m_2}+\Phi_{m_1,m_2-1}, \quad m_1>m_2>0,
\label{eq:rec1}
\end{equation}
with $\Phi_{m,0}=u_2^m$, $m\ge 0$. Furthermore, 
\begin{equation}
\Phi_{m,m}=2w\Phi_{m,m-1},\quad m\ge 1.
\label{eq:rec2}
\end{equation}

\smallskip 

Next, we introduce additional generating functions
\[
F(z,x) := F(z, x, u_{1}, u_{2}, w) = \sum_{m_1\ge m_2\ge 0}\Phi_{m_1,m_2}z^{m_1}x^{m_1-m_2}.
\]
We use the additional notation
\[
F_0(z)=F(z,0)=\sum_{m\ge 0}\Phi_{m,m}z^{m},
\]
as well as
\[
F_1(z)=\sum_{m\ge 1}\Phi_{m,m-1}z^{m}.
\]

\begin{prop}
The generating function $F(z,x)$ of weighted multivariate probability generating functions
satisfies the functional equation
\begin{equation}\label{eqn:Fzx_FEQ}
\Big(1-u_1z x-\frac1x \Big)F(z,x)=\frac{(u_2-u_1)zx}{1-u_2zx}+F_0(z)\Big(1-\frac1x\Big)-F_1(z),
\end{equation}
and
\[
F_0(z)=1+2wF_1(z).
\]
\end{prop}
\begin{proof}
The basic recurrence relation~\eqref{eq:rec1} is readily translated into a functional equation for $F(z,x)$. Summation leads to the equation
\begin{align*}
\sum_{m_1>m_2>0}\Phi_{m_1,m_2}z^{m_1}x^{m_1-m_2}&=u_1\sum_{m_1>m_2>0}\Phi_{m_1-1,m_2}z^{m_1}x^{m_1-m_2}\\
&\quad+\sum_{m_1>m_2>0}\Phi_{m_1,m_2-1}z^{m_1}x^{m_1-m_2}.
\end{align*}
After standard arguments, index shifts and splitting the sum, the left hand side evaluates to
\[
F(z,x)-F_0(z)-\frac{u_2zx}{1-u_2zx}.
\]
The two sums on the right hand side simplify as follows:
\[
u_1\sum_{m_1>m_2>0}\Phi_{m_1-1,m_2}z^{m_1}x^{m_1-m_2}=
u_1zxF(z,x)-\frac{u_1zx}{1-u_2zx}
\]
and 
\[
\sum_{m_1>m_2>0}\Phi_{m_1,m_2-1}z^{m_1}x^{m_1-m_2}=
\frac1x\big(F(z,x)-F_0(z)\big)-F_1(z).
\]
Finally, we note that ~\eqref{eq:rec2} directly implies $F_0(z)=1+2wF_1(z)$. Simplifications then lead to the stated functional equation.
\end{proof}

\begin{lem}\label{lem:GF_Fzx}
The generating function $F(z,x)$ of weighted multivariate probability generating functions
is given by
\[
F(z,x)=\frac{1}{1-x+u_1z x^2}\bigg(\frac{(u_1-u_2)zx^2}{1-u_2zx}+1-x+F_1(z)\big(2w(1-x)+x\big)\bigg),
\]
with 
\[
F_1(z)=\frac{1}{2w(1-\lambda)+\lambda}\cdot\Big(\frac{(u_2-u_1)z\lambda^2}{1-u_2z\lambda}+\lambda-1\Big)
\]
and
\[
\lambda := \lambda(z,u_1)=\frac{1-\sqrt{1-4zu_1}}{2u_1z}.
\]
\end{lem}

\begin{proof}
We use a basic application of the so-called kernel method, see, e.g., \cite{BanderierFlajolet2002,Prodinger2004}. First, we multiply the whole equation~\eqref{eqn:Fzx_FEQ} by $(-x)$.
The kernel $K(z,x,u_1)=1-x+u_1z x^2$ is canceled by the power series $x=\lambda=\lambda(z,u_1)$ stated in the formulation of the lemma:  
\[
K\big(z,\lambda(z,u_1),u_1\big)=0.
\]
Thus, by plugging $x=\lambda$ into the equation, the left-hand side and so also the right-hand side vanish, which implies the additional relation 
\[
0=\frac{(u_1-u_2)z\lambda^2}{1-u_2z\lambda}+\big(1-\lambda\big)F_0(z)+\lambda F_1(z).
\]
Using $F_0(z)=1+2wF_1(z)$ then directly leads to the stated results.
\end{proof}

Due to the dependence of the r.v.\ $L_{m_1,m_2}$ and $T_{m_1,m_2}$ stated in \eqref{eqn:L-T-Relation}, we may restrict ourselves to a joint analysis of $T_{m_1,m_2}$ and $W_{m_1,m_2}$. For extracting coefficients, we find it convenient to introduce the generating function
\begin{multline}
  \hat{F}(z,y) := \hat{F}(z,y,u,w) = \sum_{m_{1} \ge m_{2} \ge 0} \binom{m_{1} + m_{2}}{m_{1}} \E\big(u^{T_{m_1,m_2}} w^{W_{m_1,m_2}}\big) z^{m_1} y^{m_2}\\
	= \sum_{m_{1} \ge 0} \sum_{m_{2}=0}^{m_{1}} \sum_{k=0}^{m_{2}} \sum_{\ell=0}^{m_{1}} \binom{m_1+m_2}{m_1} \P\big\{W_{m_1,m_2}=k \wedge T_{m_1,m_2}=\ell \big\} w^{k} u^{\ell} z^{m_{1}} y^{m_{2}},\label{eqn:Fzy_Def}
\end{multline}
which is obtained from $F(z,x,u_{1},u_{2},w)$ via $\hat{F}(z,y,u,w) = F(zy, y^{-1}, 1, u, w)$. Note that the natural restrictions $0 \le \ell \le m_1$, $0 \le k \le m_2$ easily follow from the geometric interpretation given in Section~\ref{sec:Geometric_interpretation}.

Starting with Lemma~\ref{lem:GF_Fzx}, we easily get the following representation of $\hat{F}(z,y)$:
\begin{equation}\label{eqn:GF_Fzy}
\begin{split}
  \hat{F}(z,y) & = \frac{1-z}{(1-y-z)(1-uz)} - \frac{y}{1-y-z}\\
	& \quad \mbox{} + \frac{uzy(2w(1-y)-1)}{(1-y-z)(1-2wP(zy))(1-uP(zy))},
\end{split}
\end{equation}
with
\begin{equation*}
  P(q) = \frac{1-\sqrt{1-4q}}{2} = \sum_{n \ge 1} \frac{1}{n} \binom{2n-2}{n-1} q^{n}
\end{equation*}
denoting the generating function of shifted Catalan numbers, satisfying the functional equation
\begin{equation*}
  P(q) = \frac{q}{1-P(q)}.
\end{equation*}

Extracting coefficients from $\hat{F}(z,y)$ yields an explicit formula for the joint distribution of $W_{m_1,m_2}$ and $T_{m_1,m_2}$ (and, in view of the stated dependency \eqref{eqn:L-T-Relation}, also of $L_{m_1,m_2}$) given next.
\begin{theorem}\label{thm:WT_exact_joint}
  The joint distribution of the random variables $W_{m_1,m_2}$ and $T_{m_1,m_2}$ is, for $0 \le m_{2} \le m_{1}$, $0 \le \ell \le m_{1}$, and $0 \le k \le m_{2}$, given as follows, outside this range the probabilities are zero, anyway.
  \begin{multline*}
	  \P\big\{W_{m_1,m_2}=k \wedge T_{m_1,m_2}=\ell\big\} =\\ 
		\begin{cases} 
		\frac{\binom{m_1+m_2-\ell-1}{m_2-1}-\binom{m_1+m_2-\ell-1}{m_1-1}}{\binom{m_1+m_2}{m_1}}, & \quad \text{$k=0$ and $m_{2}, \ell \ge 1$},\\[2ex]
		\frac{2^{k} \binom{m_1+m_2-\ell-k}{m_1-1} - 2^{k+1} \binom{m_1+m_2-\ell-k-1}{m_1-1}}{\binom{m_1+m_2}{m_1}}, & \quad \text{$\ell,k \ge 1$ and $(m_1,m_2) \neq (\ell,k)$},\\[2ex]
		1, & \quad \text{$k=m_{2}=0$ and $\ell = m_{1}$},\\[0.5ex]
		1, & \quad \text{$m_{1} = m_{2} = k = \ell = 1$},\\[0.5ex]
		0, & \quad \text{otherwise}.
		\end{cases}
	\end{multline*}
\end{theorem}
\begin{proof}
A key r\^{o}le in our approach to get sum-free expressions when extracting coefficients from equation~\eqref{eqn:GF_Fzy} plays the following split of the bivariate generating function $\frac{1}{1-y-z} = \sum_{m_{1}, m_{2} \ge 0} f_{m_1,m_2} z^{m_{1}} y^{m_{2}}$, with $f_{m_1,m_2} = \binom{m_1+m_2}{m_1}$, into generating functions $S^{+}(z,y)$ and $S^{-}(z,y)$, where $S^{+}(z,y)$ only contains the terms $f_{m_1,m_2} z^{m_1} y^{m_2}$ with $m_{1} \ge m_{2} \ge 0$, and $S^{-}(z,y)$ only the ones with $m_{2} > m_{1} \ge 0$:
\begin{equation}
  \frac{1}{1-y-z} = \underbrace{\frac{1}{(1-2P(zy)) \big(1-\frac{z}{1-P(zy)}\big)}}_{=: S^{+}(z,y)} + \underbrace{\frac{P(zy)}{z(1-2P(zy))\big(1-\frac{P(zy)}{z}\big)}}_{=: S^{-}(z,y)}.
\end{equation}
This representation can be obtained easily by setting $q:=zy$ and carrying out a partial fraction expansion of the resulting factorization:
\begin{align*}
  \frac{1}{1-y-z} & = \frac{1}{1-\frac{q}{z}-z} = \frac{-z}{z^{2}-z+q} = \frac{-z}{(z-P(q))\big(z-(1-P(q))\big)}\\
	& = \frac{1}{(1-2P(q))\big(1-\frac{z}{1-P(q)}\big)} + \frac{P(q)}{z(1-2P(q))\big(1-\frac{P(q)}{z}\big)}.
\end{align*}
Since we are interested in the coefficients of $z^{m_1} y^{m_2}$, with $m_{1} \ge m_{2} \ge 0$, of the generating functions occurring, we may reduce the task of extracting coefficients to expressions, where we replace $\frac{1}{1-y-z}$ by $S^{+}(z,y)$. In particular, we use this replacement together with an application of the formal residue calculus (or alternatively, Cauchy's integration formula), see, e.g., \cite{GouldenJackson1983}, to extract coefficients of the following expression.

Let $m_{1} \ge m_{2} \ge b+r \ge b \ge a \ge 0$. Setting $q=zy$ and also taking into account $q=P(q)(1-P(q))$ and $\frac{dq}{dP} = 1-2P$, we get:
\begin{align}
  & [z^{m_1-a} y^{m_2-b}] \frac{P(zy)^{r}}{1-y-z}  = [z^{m_1-m_2+b-a} q^{m_2-b}] \frac{P(q)^{r}}{1-y-z}\notag\\
	& \quad = [z^{m_1-m_2+b-a} q^{m_2-b}] \frac{P(q)^{r}}{(1-2P(q))\big(1-\frac{z}{1-P(q)}\big)}\notag\\
	& \quad = [q^{m_2-b}] \frac{P(q)^{r}}{(1-2P(q)) \big(1-P(q)\big)^{m_1-m_2+b-a}}\notag\\
	& \quad = [q^{-1}] \frac{1}{P(q)^{m_2-b+1-r} (1-2P(q)) (1-P(q))^{m_1-a+1}}\notag\\
	& \quad = [P^{-1}] \frac{1-2P}{P^{m_2-b+1-r} (1-2P)(1-P)^{m_1-a+1}} = [P^{m_2-b-r}] \frac{1}{(1-P)^{m_1-a+1}}\notag\\
	& \quad = \binom{m_1+m_2-a-b-r}{m_1-a}.\label{eqn:coeff_extraction1}
\end{align}

With the latter result, we are ready to extract coefficients $[z^{m_1} y^{m_2} w^{k} u^{\ell}]$ from \eqref{eqn:GF_Fzy}, where we recall the natural restrictions $m_{1} \ge m_{2} \ge 0$, $0 \le \ell \le m_{1}$, and $0 \le k \le m_{2}$. We distinguish between the cases $k \ge 1$ and $k=0$.

For $k \ge 1$, it necessarily also must hold $\ell \ge 1$ in order to get non-zero coefficients, where we obtain
\begin{align*}
  [u^{\ell}w^{k}] \hat{F}(z,y,u,w) & = [u^{\ell}w^{k}] \frac{uzy(2w(1-y)-1)}{(1-y-z)(1-2wP(zy))(1-uP(zy))}\\
	& = [w^{k}] \frac{zy(2w(1-y)-1)P(zy)^{\ell-1}}{(1-y-z)(1-2wP(zy))}\\
	& = \frac{2^k zy}{1-y-z} \Big((1-y)P(zy)^{\ell+k-2} - P(zy)^{\ell+k-1} \Big),
\end{align*}
thus
\begin{multline*}
  [z^{m_1}y^{m_2}u^{\ell}w^{k}] \hat{F}(z,y,u,w) = 2^{k} [z^{m_1-1}y^{m_2-1}] \frac{P(zy)^{\ell+k-2}}{1-y-z}\\
	\mbox{} - 2^{k} [z^{m_1-1}y^{m_2-2}] \frac{P(zy)^{\ell+k-2}}{1-y-z} - 2^{k} [z^{m_1-1}y^{m_2-1}] \frac{P(zy)^{\ell+k-1}}{1-y-z}.
\end{multline*}
Using \eqref{eqn:coeff_extraction1}, one gets the explicit formula
\begin{equation}\label{eqn:Fzy_coeff_extract1}
  [z^{m_1}y^{m_2}u^{\ell}w^{k}] \hat{F}(z,y,u,w) = 2^{k} \binom{m_1+m_2-\ell-k}{m_1-1} - 2^{k+1} \binom{m_1+m_2-\ell-k-1}{m_1-1},
\end{equation}
which holds for $1 \le \ell \le m_{1}$, $1 \le k \le m_{2}$ and $(m_1,m_2) \neq (\ell,k)$. For $(m_1,m_2)=(\ell,k)$ one can check easily that $[z^{1}y^{1}u^{1}w^{1}] \hat{F}(z,y,u,w) = 2$, and zero, otherwise.

For $k=0$, we first obtain
\begin{align*}
  [w^{0}]\hat{F}(z,y,u,w) & = \hat{F}(z,y,u,0)\\
	& = \frac{1-z}{(1-y-z)(1-uz)} - \frac{y}{1-y-z} - \frac{uzy}{(1-y-z)(1-uP(zy))}.
\end{align*}
We further distinguish between $\ell=0$ and $\ell \ge 1$. The case $\ell=0$ yields
\begin{equation*}
  [u^{0}w^{0}]\hat{F}(z,y,u,w) = \hat{F}(z,y,0,0) = 1,
\end{equation*}
thus $[z^{m_1}y^{m_2}u^{0}w^{0}]\hat{F}(z,y,u,w) = 1$, for $m_1=m_2=0$, and the coefficients are zero, otherwise.

For the case $\ell \ge 1$, we proceed with
\begin{equation*}
  [u^{\ell}w^{0}]\hat{F}(z,y,u,w) = \frac{(1-z) z^{\ell}}{1-y-z} - \frac{zy P(zy)^{\ell-1}}{1-y-z}.
\end{equation*}
Since
\begin{align*}
  [z^{m_1}y^{m_2}]\frac{(1-z)z^{\ell}}{1-y-z} & = [z^{m_1-\ell}y^{m_2}]\frac{1}{1-\frac{y}{1-z}} = [z^{m_1-\ell}] \frac{1}{(1-z)^{m_2}}\\
	& = \binom{m_1+m_2-\ell-1}{m_2-1},
\end{align*}
and, by an application of \eqref{eqn:coeff_extraction1},
\begin{equation*}
  [z^{m_1}y^{m_2}]\frac{zy P(zy)^{\ell-1}}{1-y-z} = \binom{m_1+m_2-\ell-1}{m_1-1},
\end{equation*}
we thus get the explicit formula
\begin{equation}\label{eqn:Fzy_coeff_extract2}
  [z^{m_1}y^{m_2}u^{\ell}w^{0}]\hat{F}(z,y,u,w) = \binom{m_1+m_2-\ell-1}{m_2-1} - \binom{m_1+m_2-\ell-1}{m_1-1},
\end{equation}
which holds for $\ell, m_{2} \ge 1$. Moreover, $[z^{m_1}y^{0}u^{m_1}w^{0}] \hat{F}(z,y,u,w) = 1$, and the coefficients are zero, otherwise.

The exact formul{\ae} for the joint distribution of $W_{m_1,m_2}$ and $T_{m_1,m_2}$ follow then from \eqref{eqn:Fzy_coeff_extract1}, \eqref{eqn:Fzy_coeff_extract2}, together with the particular cases mentioned, by taking into account the defining equation~\eqref{eqn:Fzy_Def}.
\end{proof}

From this explicit result the marginal distributions can be obtained easily.
\begin{coroll}\label{cor:Marginal_dist}
  The exact distributions of the random variables $T_{m_1,m_2}$ and $W_{m_1,m_2}$, respectively, are, for $0 \le m_{2} \le m_{1}$, and $0 \le \ell \le m_{1}$ or $0 \le k \le m_{2}$, respectively, given as follows, outside this range the probabilities are zero, anyway.
  \begin{align*}
	  \P\big\{T_{m_1,m_2}=\ell\big\} & = 
		\begin{cases} 
		\frac{\binom{m_1+m_2-\ell-1}{m_2-1}+\binom{m_1+m_2-\ell-1}{m_1-1}}{\binom{m_1+m_2}{m_1}}, & \quad \text{$m_{2}, \ell \ge 1$},\\[2ex]
		1, & \quad \text{$m_{2}=0$ and $\ell = m_{1}$},\\[0.5ex]
		0, & \quad \text{otherwise}.
		\end{cases}\\
		\P\big\{W_{m_1,m_2}=k\big\} & = 
		\begin{cases} 
		\frac{2^{k} \binom{m_1+m_2-k}{m_1} - 2^{k+1} \binom{m_1+m_2-k-1}{m_1}}{\binom{m_1+m_2}{m_1}}, & \quad \text{$m_{1} \ge 1$},\\[2ex]
		1, & \quad \text{$m_{1} = m_{2} = k = 0$},\\[0.5ex]
		0, & \quad \text{otherwise}.
		\end{cases}
	\end{align*}
\end{coroll}
\begin{proof}
  The stated results for the marginal distributions follow from Theorem~\ref{thm:WT_exact_joint} by summation. For the probability mass function of $W_{m_1,m_2}$ one just has to apply the basic summation formula $\sum_{\ell=0}^{n} \binom{\ell}{m} = \binom{n+1}{m+1}$, whereas the corresponding result for $T_{m_1,m_2}$ follows from the fact that the sum running over $k$ telescopes. 
\end{proof}

\section{Limit laws}

\subsection{Limiting behaviour of the marginal distributions}
We first state the limiting behaviour of the individual r.v.\ $W_{m_1,m_2}$ and $T_{m_1,m_2}$, respectively, for $m_{1} \to \infty$ and depending on the growth behaviour of $m_{2}$. We exclude the case $m_{2}=0$, since then $W_{m_1,m_2}$ and $T_{m_1,m_2}$ have deterministic distributions.
\begin{theorem}\label{thm:T_limit}
The limiting distribution behaviour of the random variable $T_{m_1,m_2}$ is, for $0 \le m_{2} \le m_{1}$ and $m_{1} \to \infty$, depending on the growth behaviour of $m_{2}$ given as follows.
\begin{itemize}
\item $m_{2} \ge 1$ fixed: suitably scaled, $T_{m_1,m_2}$ converges in distribution to a r.v.~ $Y^{[m_2]}$, which is characterized via the distribution function
\begin{equation*}
  F(y) = 1-(1-y)^{m_{2}}, \quad 0 < y < 1.
\end{equation*}

It holds that $Y^{[m_{2}]}$ is Beta distributed with parameters $1$ and $m_2$, or alternatively, distributed as the minimum of $m_{2}$ independent uniform distributed r.v.\ on the interval $[0,1]$:
\begin{gather*}
  \frac{T_{m_1,m_2}}{m_{1}} \claw Y^{[m_2]}, \quad Y^{[m_{2}]} \law \BetaDist(1,m_2) \law \min\{U_{1}, U_{2}, \dots, U_{m_2}\},\\
	\text{where $U_{i} \law \Unif[0,1]$, $1 \le i \le m_{2}$, and the $U_{i}$ mutually independent}.
\end{gather*}

\item $m_{2} = o(m_{1})$, but $m_{2} \to \infty$: suitably scaled, $T_{m_{1},m_{2}}$ is asymptotically exponential distributed with parameter $1$:
\begin{gather*}
  \frac{m_{2} \cdot T_{m_1,m_2}}{m_{1}} \claw Y, \quad Y \law \Exp(1),\\
	\text{i.e., $Y$ has the distribution function $F(y) = 1-e^{-y}$, $y > 0$}.
\end{gather*}

\item $m_{2} \sim \rho \, m_{1}$, with $0 < \rho \le 1$: $T_{m_1,m_2}$ converges in distribution to a discrete r.v.\ $Y_{\rho}$, whose probability mass function is given as follows:
\begin{equation*}
  p_{\ell}^{[Y]} := \P\{Y_{\rho} = \ell\} = \frac{\rho+\rho^{\ell}}{(1+\rho)^{\ell+1}}, \quad \ell \ge 1.
\end{equation*}
It holds that $Y_{\rho}$ is the mixture of two geometrically distributed random variables:
\begin{gather*}
  Y_{\rho} \law I_{1} \cdot \Geom\Big(\frac{\rho}{1+\rho}\Big) + (1-I_{1}) \cdot \Geom\Big(\frac{1}{1+\rho}\Big),\\
	\text{with $I_{1} \law$ $\Bernoulli\Big(\frac{1}{1+\rho}\Big)$}.
\end{gather*}
\end{itemize}
\end{theorem}

\begin{remark}
Note that in the case $m_{2} \sim m_{1}$, above characterization yields a mixture of two identically distributed geometric random variables, thus for $\rho = 1$ it simply holds
\begin{equation*}
  T_{m_1,m_2} \claw Y_{1}, \quad \text{with $Y_{1} \law$ $\Geom\big(\textstyle{\frac{1}{2}}\big)$}.
\end{equation*}
\end{remark}

\begin{remark}
As discussed at the beginning, the limit laws of $T_{m_1,m_2}$ closely resemble the limit laws of the sampling without replacement urn absorbed at the boundaries~\cite{KubPanClass}, as one might expect. A heuristic explanation for the limit laws: for small $m_{2}$, i.e. $m_2 \ge 1$ fixed or $m_{2} = o(m_{1})$, but $m_{2} \to \infty$, the influence of the diagonal $y=x$ is negligible, but it gets more significant once $m_2$ and $m_1$ are of comparable size. In contrast to sampling without replacement, the limit law of $T_{m_1,m_2}$ cannot degenerate (which happens for the sampling urn, if $m_2 \gg m_1$), as we stay in the cone $x\ge y\ge 0$. This heuristic can be made precise, since $T_{m_1,m_2}$ is distributed as the distance from the origin, where a particle, starting at $(m_1,m_2)$ and moving according to a sampling without replacement urn, either hits the $x$-axis or the $y$-axis.
\end{remark}

\begin{theorem}\label{thm:W_limit}
The limiting distribution behaviour of the random variable $W_{m_1,m_2}$ is, for $0 \le m_{2} \le m_{1}$ and $m_{1} \to \infty$, depending on the growth behaviour of $m_{2}$ given as follows.
\begin{itemize}
\item $m_{2} = o(m_{1})$: $W_{m_1,m_2}$ has a degenerate limiting distribution, $W_{m_1,m_2} \claw\: 0$.

\item $m_{2} \sim \rho \, m_{1}$, with $0 < \rho < 1$: $W_{m_1,m_2}$ converges in distribution to a (shifted) geometrically distributed r.v.\ $X_{\rho}$ with success probability $\frac{1-\rho}{1+\rho}$:
\begin{equation*}
  W_{m_1,m_2} \claw X_{\rho}, \quad \text{with $X_{\rho} \law$ $\Geom\Big(\frac{1-\rho}{1+\rho}\Big)-1$},
\end{equation*}
thus with probability mass function
\begin{equation*}
  p_{k}^{[X]} := \P\{X_{\rho} = k\} = \frac{1-\rho}{1+\rho} \cdot \Big(\frac{2 \rho}{1+\rho}\Big)^{k}, \quad k \ge 0.
\end{equation*}

\item $m_{2} \sim m_{1}$, where the difference $d := m_{1}-m_{2}$ satisfies $\sqrt{m_{1}} \ll d \ll m_{1}$: suitably scaled, $W_{m_{1},m_{2}}$ is asymptotically exponential distributed with parameter $\frac{1}{2}$:
\begin{gather*}
  \frac{d \cdot W_{m_1,m_2}}{m_{1}} \claw X, \quad X \law \Exp\big(\textstyle{\frac{1}{2}}\big),\\
	\text{i.e., $X$ has the distribution function $F(x) = 1-e^{-\frac{x}{2}}$, $x > 0$}.
\end{gather*}

\item $m_{2} \sim m_{1}$, where the difference $d := m_{1}-m_{2}$ satisfies $d \sim \alpha \sqrt{m_{1}}$, with $\alpha > 0$: suitably scaled, $W_{m_{1},m_{2}}$ converges in distribution to a r.v.\ $X_{\alpha}$, which is characterized via the distribution function
\begin{equation*}
  F(x) = 1-e^{-\frac{x(2\alpha+x)}{4}}, \quad x > 0.
\end{equation*}

Thus, $W_{m_{1},m_{2}}$ is asymptotically linear exponential distributed:
\begin{equation*}
  \frac{W_{m_1,m_2}}{\sqrt{m_{1}}} \claw X_{\alpha}, \quad X_{\alpha} \law \LinExp\big(\textstyle{\frac{\alpha}{2}},\textstyle{\frac{1}{2}}\big).
\end{equation*}

\item $m_{2} \sim m_{1}$, where the difference $d := m_{1}-m_{2}$ satisfies $d = o(\sqrt{m_{1}})$: suitably scaled, $W_{m_{1},m_{2}}$ converges in distribution to a Rayleigh distributed r.v.\ $X$ with parameter $\sqrt{2}$:
\begin{gather*}
  \frac{W_{m_1,m_2}}{\sqrt{m_{1}}} \claw X, \quad X \law \Rayleigh\big(\sqrt{2}\big),\\
	\text{thus $X$ has the distribution function $F(x) = 1-e^{-\frac{x^{2}}{4}}$, $x > 0$}.
\end{gather*}
\end{itemize}
\end{theorem}

\begin{remark}[Properties of linear exponential distributions]
The so-called linear exponential distribution $\LinExp(\lambda,\nu)$, see \cite[p.~480]{JohnsonKotzBala1994}, is a two-parametric family of continuous distributions, with support $\mathbb{R}^{+}$, characterized via the distribution function
\begin{equation*}
  F(x) = 1-e^{-\big(\lambda x + \frac{\nu x^{2}}{2}\big)}, \quad x > 0.
\end{equation*}
It generalizes the exponential distribution as well as the Rayleigh distribution.
Alternatively, $\LinExp(\lambda,\nu)$ can be represented in term of an Exponential distribution $Z \law \text{Exp}(1)$:
\[
Y\law -\frac{\lambda}{\mu} + \frac{\sqrt{\lambda^2+2\mu Z}}{\mu}.
\]
\end{remark}

\begin{remark}
The range $m_{2} \sim m_{1}$ and $d := m_{1}-m_{2}\sim \alpha \sqrt{m_1}$ can be relaxed to
$d := m_{1}-m_{2}\sim \alpha_{m_1} \sqrt{m_1}$, allowing also sequences $\alpha_{m_1}\to 0$, 
as the linear exponential distribution degenerates to a Rayleigh law.
\end{remark}

Our results in Lemmata~\ref{lem:P-W-Relation} and~\ref{lem:L-T-C-Relation} combined with Theorem~\ref{thm:W_limit} 
allows one to directly give alternative proofs for the distribution as well as limit laws for the total number of correct guesses $C_{m_1,m_2}$. 
As mentioned before, these limit laws were stated in~\cite{KuPanPro2009}, obtained from the explicit distribution of $C_{m_1,m_2}$ given in~\cite{Sulanke,KnoPro2001},
\[
\P\{C_{m_1,m_2}=k\}=\frac{\binom{m_1+m_2}{k}-\binom{m_1+m_2}{k+1}}{\binom{m_1+m_2}{m_1}},\quad m_1\le k \le m_1+m_2.
\]
but without proofs. Here, we only additionally note in passing a very simple relation between the moments of $C_{m_1,m_2}$ and $W_{m_1,m_2}$.

\begin{coroll}[Moments of the total number of correct guesses]
\label{coro}
Let $\hat{C}_{m_1,m_2}$ denote the shifted number of correct guesses $C_{m_1,m_2}-m_1$.
The $s$-th factorial moments $\E(\fallfak{\hat{C}_{m_1,m_2}}{s})$ of $\hat{C}_{m_1,m_2}$ satisfy
\[
\E(\fallfak{\hat{C}_{m_1,m_2}}{s})=\frac1{2^s}\E(\fallfak{W_{m_1,m_2}}{s}),\quad s\ge 1.
\]
\end{coroll}

\begin{proof}[Proof of Theorem~\ref{thm:T_limit}-\ref{thm:W_limit}]
The stated limiting distribution results are obtained from the explicit formul{\ae} given in Corollary~\ref{cor:Marginal_dist} via a case-by-case study according to the growth behaviour of $m_2$ and a careful application of Stirling's formula for the factorials, 
\begin{equation}\label{eqn:Stirling_formula}
  \ln n! = \big(n+\textstyle{\frac{1}{2}}\big) \ln n - n + \textstyle{\frac{1}{2}} \ln(2 \pi) + \mathcal{O}\big(n^{-1}\big).
\end{equation}
Since these computations are analogous to the ones for the joint limit laws given in Section~\ref{sec:Joint_limit_laws} we omit them here. 
\end{proof}

\begin{proof}[Proof of Corollary~\ref{coro}]
By well-known properties of the binomial distribution (or via direct computation) we obtain for integers $s\ge 1$:
\[
\E(\fallfak{\hat{C}_{m_1,m_2}}{s}\mid W_{m_1,m_2}) = \frac{1}{2^s}\fallfak{W_{m_1,m_2}}{s}.
\]
By the tower rule of total expectation,
\[
\E(\fallfak{\hat{C}_{m_1,m_2}}{s})=\E\Big(\E(\fallfak{\hat{C}_{m_1,m_2}}{s}\mid W_{m_1,m_2})\Big),
\]
the stated result follows.
\end{proof}

\subsection{Joint limit laws}\label{sec:Joint_limit_laws}

From the explicit formul{\ae} for the joint distribution of $W_{m_1,m_2}$ and $T_{m_1,m_2}$ given in Theorem~\ref{thm:WT_exact_joint}, we deduce the following joint limit laws.
\begin{theorem}\label{thm:WT-JointLimit}
The joint limiting distribution behaviour of the random variables $W_{m_1,m_2}$ and $T_{m_1,m_2}$ is, for $0 \le m_{2} \le m_{1}$ and $m_{1} \to \infty$, depending on the growth behaviour of $m_{2}$ given as follows.
\begin{itemize}
\item $m_{2} \sim \rho m_{1}$, with $0 < \rho < 1$: $(W_{m_1,m_2}, T_{m_1,m_2})$ converge in distribution to a pair of discrete r.v.\ $(\hat{X}_{\rho}, \hat{Y}_{\rho})$, whose joint probability mass function is, for $k \ge 0$ and $\ell \ge 1$, given as follows:
\begin{equation*}
  \hat{p}_{k,\ell} := \P\{\hat{X}_{\rho} = k \wedge \hat{Y}_{\rho} = \ell\} = 
	\begin{cases}
	  \frac{\rho-\rho^{\ell}}{(1+\rho)^{\ell+1}}, & \quad \text{for $k=0$ and $\ell \ge 1$},\\
		\frac{1-\rho}{\rho(1+\rho)} \cdot \left(\frac{\rho}{1+\rho}\right)^{\ell} \cdot \left(\frac{2\rho}{1+\rho}\right)^{k}, & \quad \text{for $k, \ell \ge 1$}.
	\end{cases}
\end{equation*}
\end{itemize}

For all other growth ranges of $m_{2}$, $W_{m_1,m_2}$ and $T_{m_1,m_2}$ are asymptotically independent, i.e., suitably scaled, this pair of r.v.\ weakly converges to a pair $(X,Y)$ of independent r.v., distributed as the limit laws of the corresponding marginal distributions given in Theorem~\ref{thm:T_limit}-\ref{thm:W_limit}, thus: 
\begin{itemize}
\item $m_{2} \ge 1$ fixed:
\begin{equation*}
  \Big(W_{m_1,m_2}, \frac{T_{m_1,m_2}}{m_1}\Big) \claw\: (X,Y) \law \big(0,\BetaDist(1,m_2)\big) \law \big(0,\min\{{\Unif}[0,1]_{i} : 1 \le i \le m_{2}\}\big).
\end{equation*}
\item $m_{2} = o(m_{1})$, but $m_{2} \to \infty$:
\begin{equation*}
  \Big(W_{m_1,m_2}, \frac{m_{2} \cdot T_{m_1,m_2}}{m_1}\Big) \claw\: (X,Y) \law \big(0,\Exp(1)\big).
\end{equation*}
\item $m_{2} \sim m_{1}$, where the difference $d := m_{1}-m_{2}$ satisfies $\sqrt{m_{1}} \ll d \ll m_{1}$:
\begin{equation*}
  \Big(\frac{d \cdot W_{m_1,m_2}}{m_1}, T_{m_1,m_2}\Big) \claw\: (X,Y) \law \Big(\Exp\big(\textstyle{\frac{1}{2}}\big), \Geom\big(\textstyle{\frac{1}{2}}\big)\Big).
\end{equation*}
\item $m_{2} \sim m_{1}$, where the difference $d := m_{1}-m_{2}$ satisfies $d \sim \alpha \sqrt{m_{1}}$, with $\alpha > 0$:
\begin{equation*}
  \Big(\frac{W_{m_1,m_2}}{\sqrt{m_1}},T_{m_1,m_2}\Big) \claw\: (X,Y) \law \Big(\LinExp\big(\textstyle{\frac{\alpha}{2}},\textstyle{\frac{1}{2}}\big), \Geom\big(\textstyle{\frac{1}{2}}\big)\Big).
\end{equation*}
\item $m_{2} \sim m_{1}$, where the difference $d := m_{1}-m_{2}$ satisfies $d = o(\sqrt{m_{1}})$:
\begin{equation*}
  \Big(\frac{W_{m_1,m_2}}{\sqrt{m_1}},T_{m_1,m_2}\Big) \claw\: (X,Y) \law \Big(\Rayleigh\big(\sqrt{2}\big), \Geom\big(\textstyle{\frac{1}{2}}\big)\Big).
\end{equation*}
\end{itemize}
\end{theorem}
\begin{proof}
  The joint limiting behaviour of $W_{m_1,m_2}$ and $T_{m_1,m_2}$, for $m_1 \to \infty$, follows from the explicit formul{\ae} for the joint probability mass function given in Theorem~\ref{thm:WT_exact_joint} by distinguishing between various cases according to the growth behaviour of $m_2$, using suitable estimates, and applying Stirling's formula~\eqref{eqn:Stirling_formula}. For some regions of $m_2$ we find it slightly more convenient to start with explicit formul{\ae} for the joint or one-sided cumulative distribution functions, which are obtained easily from Theorem~\ref{thm:WT_exact_joint} by summation, thus, we just state them. Namely, for $1 \le \ell \le m_1$, $0 \le k \le m_2$ and $(m_1,m_2) \neq (\ell,k)$, one gets
	\begin{multline}\label{eqn:W-T-CDF}
	  \P\{W_{m_1,m_2} \le k \wedge T_{m_1,m_2} \le \ell\}\\
		=	1-\frac{\binom{m_1+m_2-\ell-1}{m_2} + \binom{m_1+m_2-\ell-1}{m_1} + 2^{k+1} \binom{m_1+m_2-k-1}{m_1} - 2^{k+1} \binom{m_1+m_2-k-\ell-1}{m_1}}{\binom{m_1+m_2}{m_1}},
	\end{multline}
	as well as (we omit stating the cases $(m_1,m_2)=(\ell,k)$ not relevant here)
	\begin{multline}\label{eqn:W-T-oneCDF}
	  \P\{W_{m_1,m_2} \le k \wedge T_{m_1,m_2}=\ell\}\\
		=	\frac{\binom{m_1+m_2-\ell-1}{m_2-1}+\binom{m_1+m_2-\ell-1}{m_1-1} - 2^{k+1} \binom{m_1+m_2-k-\ell-1}{m_1-1}}{\binom{m_1+m_2}{m_1}}.
	\end{multline}
  
	Next we sketch the asymptotic considerations for the different cases.
	\begin{itemize}
	\item $m_2 \ge 1$ fixed: setting $k=0$ in \eqref{eqn:W-T-CDF} yields
	\begin{equation}\label{eqn:W-T-CDF-k0}
	  \P\{W_{m_1,m_2} =0 \wedge T_{m_1,m_2} \le \ell\} = 1- \frac{\binom{m_1+m_2-\ell-1}{m_2}}{\binom{m_1+m_2}{m_2}}
		- \frac{2 \binom{m_1+m_2-1}{m_1}}{\binom{m_1+m_2}{m_1}} + \frac{\binom{m_1+m_2-\ell-1}{m_1}}{\binom{m_1+m_2}{m_1}}.
	\end{equation}
	Since
	\begin{align*}
	  \frac{\binom{m_1+m_2-1}{m_1}}{\binom{m_1+m_2}{m_1}} & = \frac{m_2}{m_1+m_2} = \mathcal{O}\big(\textstyle{\frac{m_2}{m_1}}\big),\\
		\frac{\binom{m_1+m_2-\ell-1}{m_1}}{\binom{m_1+m_2}{m_1}} & = \frac{m_{2}^{\underline{\ell+1}}}{(m_1+m_2)^{\underline{\ell+1}}} = 0, \quad \text{for $\ell \ge m_2$},
	\end{align*}
	the asymptotic contribution relevant for $m_2$ fixed is stemming from
	\begin{align*}
	  \frac{\binom{m_1+m_2-\ell-1}{m_2}}{\binom{m_1+m_2}{m_2}} & = \frac{m_{1}^{\underline{\ell+1}}}{(m_1+m_2)^{\underline{\ell+1}}} = \frac{(m_1+m_2-\ell-1)^{\underline{m_2}}}{(m_1+m_2)^{\underline{m_2}}}\\
		& = \frac{(m_1-\ell)^{m_2} \cdot \big(1+\mathcal{O}\big(\frac{1}{m_1-\ell}\big)\big)}{m_1^{m_2} \cdot \big(1+\mathcal{O}\big(\frac{1}{m_1}\big)\big)} = \Big(1-\frac{\ell}{m_1}\Big)^{m_2} \cdot \Big(1+\mathcal{O}\big(\textstyle{\frac{1}{m_1-\ell}}\big)\Big).
	\end{align*}
	Thus, for $m_2 \ge 1$ fixed, $m_1 \to \infty$, and, e.g., $\sqrt{m_1} \le \ell \le m_1-\sqrt{m_1}$, it holds
	\begin{equation*}
	  \P\{W_{m_1,m_2} =0 \wedge T_{m_1,m_2} \le \ell\} \to 1- \Big(1-\frac{\ell}{m_1}\Big)^{m_2},
	\end{equation*}
	i.e., by setting $y=\frac{\ell}{m_1}$,
	\begin{equation*}
	  \P\Big\{W_{m_1,m_2}=0 \wedge \frac{T_{m_1,m_2}}{m_1} \le y\Big\} \to F(y) = 1-(1-y)^{m_2}, \quad \text{for $0 < y < 1$}.
	\end{equation*}
	$F(y)$ is the distribution function of a Beta distribution or, alternatively, the distribution function of the minimum of $m_2$ independent uniformly on $[0,1]$ distributed r.v.
	\item $m_2 =o(m_1)$, but $m_2 \to \infty$: again we consider \eqref{eqn:W-T-CDF-k0}, where, as a refinement to the considerations of the previous case, we get by assuming $\frac{m_1}{m_2} \ge 2$ due to simple estimates
	\begin{equation*}
	\frac{\binom{m_1+m_2-\ell-1}{m_1}}{\binom{m_1+m_2}{m_1}} \le \Big(\frac{m_2}{m_1+m_2}\Big)^{\ell+1}
	= e^{-(\ell+1)\ln\big(1+\frac{m_2}{m_1}\big)} \le e^{-\ell}, \quad \text{for $\ell \ge 1$},
	\end{equation*}
	thus asymptotically negligible contributions for $\ell \to \infty$. Again, the relevant contribution for the asymptotic behaviour is coming from
	\begin{align*}
	  \frac{\binom{m_1+m_2-\ell-1}{m_2}}{\binom{m_1+m_2}{m_2}} & = \frac{m_1! \cdot (m_1+m_2-\ell-1)!}{(m_1-\ell-1)! \cdot (m_1+m_2)!}\\
		& = e^{-\frac{\ell m_2}{m_1}} \cdot \Big(1+\mathcal{O}\big(\textstyle{\frac{\ell m_2^{2}}{m_1^{2}}}\big)+\mathcal{O}\big(\textstyle{\frac{\ell ^{2} m_2}{m_1^{2}}}\big)+\mathcal{O}\big(\textstyle{\frac{\ell}{m_1}}\big)+\mathcal{O}\big(\textstyle{\frac{m_2}{m_1}}\big) \Big),
	\end{align*}
	where the asymptotic evaluation is obtained by applying Stirling's formula.
	Thus, for $m_1, m_2 \to \infty$, $m_2 = o(m_1)$, and say $\sqrt{\frac{m_1}{m_2}} \le \ell \le \min\Big\{\big(\frac{m_1}{m_2}\big)^{\frac{3}{2}},\frac{m_1}{m_2^{\frac{3}{4}}}\Big\}$, it holds
	\begin{equation*}
	  \P\{W_{m_1,m_2} =0 \wedge T_{m_1,m_2} \le \ell\} \to 1 - e^{-\frac{\ell m_2}{m_1}},
	\end{equation*}
	i.e., by setting $y=\frac{\ell m_2}{m_1}$,
	\begin{equation*}
	  \P\Big\{W_{m_1,m_2}=0 \wedge \frac{m_2 \cdot T_{m_1,m_2}}{m_1} \le y\Big\} \to 1-e^{-y}, \quad \text{for $y > 0$}.
	\end{equation*}
	\item $m_{2} \sim \rho m_{1}$, with $0 < \rho < 1$: in this range one considers the asymptotic behaviour of the probability mass function given in Theorem~\ref{thm:WT_exact_joint}, for fixed $k \ge 0$, $\ell \ge 1$. The case $k=0$ has to be treaten separately, where we obtain
	\begin{align*}
	  \frac{\binom{m_1+m_2-\ell-1}{m_2-1}}{\binom{m_1+m_2}{m_2}} & = \frac{m_2 m_1^{\underline{\ell}}}{(m_1+m_2)^{\underline{\ell+1}}} = \frac{\frac{m_2}{m_1}}{\big(1+\frac{m_2}{m_1}\big)^{\ell+1}} \cdot \Big(1+\mathcal{O}\big(\textstyle{\frac{1}{m_1}}\big)\Big),\\
		\frac{\binom{m_1+m_2-\ell-1}{m_1-1}}{\binom{m_1+m_2}{m_1}} & = \frac{m_1 m_2^{\underline{\ell}}}{(m_1+m_2)^{\underline{\ell+1}}} = \frac{(\frac{m_2}{m_1})^{\ell}}{\big(1+\frac{m_2}{m_1}\big)^{\ell+1}} \cdot \Big(1+\mathcal{O}\big(\textstyle{\frac{1}{m_2}}\big)\Big).
	\end{align*}
	Thus, for $m_1 \to \infty$, $m_2 \sim \rho m_1$, with $0 < \rho < 1$, and $\ell \ge 1$ fixed, it holds
	\begin{equation*}
	  \P\{W_{m_1,m_2} = 0 \wedge T_{m_1,m_2} = \ell\} \to \frac{\rho - \rho^{\ell}}{(1+\rho)^{\ell+1}}.
	\end{equation*}
	Similarly, for $k \ge 1$ one gets
	\begin{align*}
	  \frac{2^{k} \binom{m_1+m_2-\ell-k}{m_1-1}}{\binom{m_1+m_2}{m_1}} & = \frac{2^{k} m_1 m_2^{\underline{\ell+k-1}}}{(m_1+m_2)^{\underline{\ell+k}}}
		= \frac{2^{k} (\frac{m_2}{m_1})^{\ell+k-1}}{\big(1+\frac{m_2}{m_1}\big)^{\ell+k}} \cdot \Big(1+\mathcal{O}\big(\textstyle{\frac{1}{m_2}}\big)\Big),\\
		\frac{2^{k+1} \binom{m_1+m_2-\ell-k-1}{m_1-1}}{\binom{m_1+m_2}{m_1}} & = \frac{2^{k+1} m_1 m_2^{\underline{\ell+k}}}{(m_1+m_2)^{\underline{\ell+k+1}}}
		= \frac{2^{k+1} (\frac{m_2}{m_1})^{\ell+k}}{\big(1+\frac{m_2}{m_1}\big)^{\ell+k+1}} \cdot \Big(1+\mathcal{O}\big(\textstyle{\frac{1}{m_2}}\big)\Big),
	\end{align*}
	which implies, for $m_1 \to \infty$, $m_2 \sim \rho m_1$, with $0 < \rho < 1$, and $k, \ell \ge 1$ fixed,
	\begin{multline*}
	  \P\{W_{m_1,m_2} = k \wedge T_{m_1,m_2} = \ell\} \to \frac{2^{k} \rho^{\ell+k-1}}{(1+\rho)^{\ell+k}} - \frac{2^{k+1} \rho^{\ell+k}}{(1+\rho)^{\ell+k+1}}\\
		= \frac{1-\rho}{\rho (1+\rho)} \cdot \Big(\frac{\rho}{1+\rho}\Big)^{\ell} \cdot \Big(\frac{2 \rho}{1+\rho}\Big)^{k}.
	\end{multline*}
	\item $m_{2} \sim m_{1}$, where the difference $d := m_{1}-m_{2}$ satisfies $\sqrt{m_{1}} \ll d \ll m_{1}$: we start with \eqref{eqn:W-T-oneCDF} and examine its asymptotic behaviour for $\ell \ge 1$ fixed, where at first we only assume $d=o(m_1)$. Analogous to previous considerations we get
	\begin{align*}
	  \frac{\binom{m_1+m_2-\ell-1}{m_2-1}}{\binom{m_1+m_2}{m_2}} & = \frac{\frac{m_2}{m_1}}{\big(1+\frac{m_2}{m_1}\big)^{\ell+1}} \cdot \Big(1+\mathcal{O}\big(\textstyle{\frac{1}{m_1}}\big)\Big) = \frac{1}{2^{\ell+1}} \cdot \Big(1+\mathcal{O}\big(\textstyle{\frac{d}{m_1}}\big)\Big),\\
		\frac{\binom{m_1+m_2-\ell-1}{m_1-1}}{\binom{m_1+m_2}{m_1}} & = \frac{(\frac{m_2}{m_1})^{\ell}}{\big(1+\frac{m_2}{m_1}\big)^{\ell+1}} \cdot \Big(1+\mathcal{O}\big(\textstyle{\frac{1}{m_1}}\big)\Big) = \frac{1}{2^{\ell+1}} \cdot \Big(1+\mathcal{O}\big(\textstyle{\frac{d}{m_1}}\big)\Big),
	\end{align*}
	and
	\begin{equation*}
	  \frac{2^{k+1} \binom{m_1+m_2-k-\ell-1}{m_1-1}}{\binom{m_1+m_2}{m_1}} = \frac{2m_1 (m_2-k)^{\underline{\ell}}}{(m_1+m_2-k)^{\underline{\ell+1}}} \cdot
		\frac{2^{k} \binom{m_1+m_2-k}{m_1}}{\binom{m_1+m_2}{m_1}}.
	\end{equation*}
	It easily follows that, for $\ell \ge 1$ fixed,
	\begin{equation*}
	  \frac{2m_1 (m_2-k)^{\underline{\ell}}}{(m_1+m_2-k)^{\underline{\ell+1}}} = \frac{2 m_1 m_2^{\ell}}{(m_1+m_2)^{\ell+1}} \cdot \Big(1+\mathcal{O}\big(\textstyle{\frac{k}{m_2}}\big)\Big) = \frac{1}{2^{\ell}} \cdot \Big(1+\mathcal{O}\big(\textstyle{\frac{d}{m_1}}\big)+\mathcal{O}\big(\textstyle{\frac{k}{m_1}}\big)\Big),
	\end{equation*}
	whereas, by an application of Stirling's formula, one obtains
	\begin{equation*}
	  \frac{2^{k} \binom{m_1+m_2-k}{m_1}}{\binom{m_1+m_2}{m_1}} = e^{- \frac{k(2d+k)}{4m_1}} \cdot \Big(1+\mathcal{O}\big(\textstyle{\frac{k d^{2}}{m_1^{2}}}\big)+\mathcal{O}\big(\textstyle{\frac{k^{3}}{m_1^{2}}}\big)\Big).
	\end{equation*}
	Thus, putting things together, we obtain, for $\ell \ge 1$ fixed, the following asymptotic result, valid uniformly for say $k \le \min\big\{m_1^{\frac{2}{3}-\epsilon},(\frac{m_1}{d})^{2-\epsilon}\big\}$, for a fixed $\epsilon > 0$,
	\begin{align}
	  & \P\{W_{m_1,m_2} \le k \wedge T_{m_1,m_2} = \ell\}\notag\\
		& \quad = 2^{-\ell} \cdot \left(1-e^{-\frac{k (2d+k)}{4m_1}}\right) \cdot \Big(1+\mathcal{O}\big(\textstyle{\frac{d}{m_1}}\big)+\mathcal{O}\big(\textstyle{\frac{k}{m_1}}\big)+\mathcal{O}\big(\textstyle{\frac{kd^{2}}{m_1^{2}}}\big)+\mathcal{O}\big(\textstyle{\frac{k^{3}}{m_1^{2}}}\big)\Big).\label{eqn:WT-asympt}
	\end{align}
	
	Now we restrict our considerations to the range $\sqrt{m_1} \ll d \ll m_1$. For $m_1 \to \infty$, $\ell \ge 1$ fixed, and say $k \le \min\Big\{\big(\frac{m_1}{d}\big)^{\frac{3}{2}},\frac{m_1^{\frac{3}{4}}}{\sqrt{d}}\Big\}$, which in addition implies that $k=o(d)$, it holds
	\begin{equation*}
	  \P\{W_{m_1,m_2} \le k \wedge T_{m_1,m_2} = \ell\} \to 2^{-\ell} \cdot \Big(1-e^{-\frac{k d}{2 m_1}}\Big),
	\end{equation*}
	i.e., by setting $x=\frac{k d}{m_1}$,
	\begin{equation*}
	  \P\Big\{\frac{d \cdot W_{m_1,m_2}}{m_1} \le x \wedge T_{m_1,m_2} = \ell\Big\} \to 2^{-\ell} \cdot \Big(1-e^{-\frac{x}{2}}\Big), \quad \text{for $x > 0$}.
	\end{equation*}
	\item $m_{2} \sim m_{1}$, where the difference $d := m_{1}-m_{2}$ satisfies $d \sim \alpha \sqrt{m_{1}}$, with $\alpha > 0$: using \eqref{eqn:WT-asympt}, we obtain, for $m_1 \to \infty$, $\ell \ge 1$ fixed, and $k = \mathcal{O}\big(m^{\frac{1}{2}+\epsilon}\big)$, with $\epsilon>0$, when setting $x=\frac{k}{\sqrt{m_1}}$,
	\begin{equation*}
	  \P\Big\{\frac{W_{m_1,m_2}}{\sqrt{m_1}} \le x \wedge T_{m_1,m_2} = \ell\Big\} \to 2^{-\ell} \cdot \Big(1-e^{-\frac{x(2\alpha+x)}{4}}\Big), \quad \text{for $x > 0$}.
	\end{equation*}
	\item $m_{2} \sim m_{1}$, where the difference $d := m_{1}-m_{2}$ satisfies $d = o(\sqrt{m_{1}})$: For $m_1 \to \infty$, $\ell \ge 1$ fixed, and say $k \le \min\Big\{m_{1}^{\frac{1}{2}+\epsilon},\frac{m_1^{\frac{3}{4}}}{\sqrt{d}}\Big\}$, it holds
	\begin{equation*}
	  \P\{W_{m_1,m_2} \le k \wedge T_{m_1,m_2} = \ell\} \to 2^{-\ell} \cdot \Big(1-e^{-\frac{k^{2}}{4m_1}}\Big),
	\end{equation*}
	i.e., by setting $x=\frac{k}{\sqrt{m_1}}$,
	\begin{equation*}
	  \P\Big\{\frac{\cdot W_{m_1,m_2}}{\sqrt{m_1}} \le x \wedge T_{m_1,m_2} = \ell\Big\} \to 2^{-\ell} \cdot \Big(1-e^{-\frac{x^{2}}{4}}\Big), \quad \text{for $x > 0$}.
	\end{equation*}
	\end{itemize}
\end{proof}

\subsection{Correlation in central region}

As a consequence of Theorem~\ref{thm:WT-JointLimit}, the discrete r.v.\ $\hat{X}_{\rho}$ and $\hat{Y}_{\rho}$ occurring as limits of $W_{m_1,m_2}$ and $T_{m_1,m_2}$, resp., in the central region $m_2 \sim \rho m_1$, $0 < \rho < 1$, are not independent. In order to get some insight into the joint variability of $\hat{X}_{\rho}$ and $\hat{Y}_{\rho}$ and their linear correlation we compute the covariance between $\hat{X}_{\rho}$ and $\hat{Y}_{\rho}$ as well as the correlation coefficient.

To this aim, let us introduce the joint moment generating function
\begin{equation*}
  \hat{P}(s,t) := \E\big(e^{s \hat{X}_{\rho} + t \hat{Y}_{\rho}}\big) = \sum_{k \ge 0} \sum_{\ell \ge 1} \P\big\{\hat{X}_{\rho}=k \wedge \hat{Y}_{\rho}=\ell\big\} e^{ks + \ell t}.
\end{equation*}
Using $\hat{p}_{k, \ell} = \P\big\{\hat{X}_{\rho}=k \wedge \hat{Y}_{\rho}=\ell\big\}$ as defined in Theorem~\ref{thm:WT-JointLimit}, we get, after distinguishing $k=0$ and $k \ge 1$, via basic summation and simple manipulations,
\begin{align*}
  \sum_{\ell \ge 1} \hat{p}_{0,\ell} e^{\ell t} & = \sum_{\ell \ge 1} \frac{\rho - \rho^{\ell}}{(1+\rho)^{\ell+1}} e^{\ell t} = \frac{\rho (1-\rho) e^{2t}}{(1+\rho)(1+\rho-e^{t})(1+\rho-\rho e^{t})},\\
	\sum_{k \ge 1} \sum_{\ell \ge 1} \hat{p}_{k,\ell} e^{k s + \ell t} & = \sum_{k \ge 1} \sum_{\ell \ge 1} \frac{1-\rho}{\rho(1+\rho)} \Big(\frac{\rho}{1+\rho}\Big)^{\ell} \Big(\frac{2\rho}{1+\rho}\Big)^{k} e^{ks + \ell t}\\
	& = \frac{2\rho (1-\rho) e^{s+t}}{(1+\rho)(1+\rho-2\rho e^{s})(1+\rho-\rho e^{t})},
\end{align*}
and, by combining these expressions, the following result for the joint moment generating function:
\begin{equation}\label{eqn:Pst_explicit}
  \hat{P}(s,t) = \frac{\rho(1-\rho)e^{t} (e^{t} +2e^{s}-2e^{s+t})}{(1+\rho-e^{t})(1+\rho-\rho e^{t})(1+\rho-2\rho e^{s})}.
\end{equation}
Carrying out a series expansion of \eqref{eqn:Pst_explicit} around $(s,t)=(0,0)$ and taking into account
\begin{equation*}
  \E\big(e^{s \hat{X}_{\rho} + t \hat{Y}_{\rho}}\big) = 1 + \E(\hat{X}_{\rho}) s + \E(\hat{Y}_{\rho}) t + \E(\hat{X}_{\rho}^{2}) \frac{s^{2}}{2} + \E(\hat{X}_{\rho} \hat{Y}_{\rho}) st + \E(\hat{Y}_{\rho}^{2}) \frac{t^{2}}{2} + \mathcal{O}\big((s+t)^{3}\big),
\end{equation*}
yields the first moments of $\hat{X}_{\rho}$ and $\hat{Y}_{\rho}$:
\begin{gather*}
  \E(\hat{X}_{\rho}) = \frac{2\rho}{1-\rho}, \quad \E(\hat{Y}_{\rho}) = \frac{\rho^{2}+1}{\rho},\\
	\E(\hat{X}_{\rho}^{2}) = \frac{2\rho (1+3\rho)}{(1-\rho)^{2}}, \quad
	\E(\hat{Y}_{\rho}^{2}) = \frac{2\rho^{4}+\rho^{3}+\rho+2}{\rho^{2}}, \quad \E(\hat{X}_{\rho} \hat{Y}_{\rho}) = \frac{2\rho (1+\rho)}{1-\rho}.
\end{gather*}
From these expressions, the variances and the covariance immediately follow:
\begin{align*}
  \Var(\hat{X}_{\rho}) & = \E(\hat{X}_{\rho}^{2}) - \big(\E(\hat{X}_{\rho})\big)^{2} = \frac{2 \rho(1+\rho)}{(1-\rho)^{2}},\\
	\Var(\hat{Y}_{\rho}) & = \E(\hat{Y}_{\rho}^{2}) - \big(\E(\hat{Y}_{\rho})\big)^{2} = \frac{1+\rho-2\rho^{2}+\rho^{3}+\rho^{4}}{\rho^{2}},\\
	\Cov(\hat{X}_{\rho},\hat{Y}_{\rho}) & = \E(\hat{X}_{\rho} \hat{Y}_{\rho}) - \E(\hat{X}_{\rho}) \E(\hat{Y}_{\rho}) = -2.
\end{align*}
Thus, the covariance $\Cov(\hat{X}_{\rho},\hat{Y}_{\rho})$ does not depend on the ratio $\rho$. Furthermore, the correlation coefficient $C_{\rho} := \Corr(\hat{X}_{\rho},\hat{Y}_{\rho})$ is given by
\begin{equation*}
  C_{\rho} = \Corr(\hat{X}_{\rho},\hat{Y}_{\rho}) = {\textstyle{\frac{\Cov(\hat{X}_{\rho},\hat{Y}_{\rho})}{\sqrt{\Var(\hat{X}_{\rho})} \, \sqrt{\Var(\hat{Y}_{\rho})}}}} = - \frac{\sqrt{2 \rho} \, (1-\rho)}{\sqrt{(1+\rho)(1+\rho-2\rho^{2}+\rho^{3}+\rho^{4})}}.
\end{equation*}
Simple calculations show that $C_{\rho}$ takes its minimum for $\rho = \tilde{\rho}$, the unique root of $\kappa(\rho) = 1-3\rho-3\rho^{2}+3\rho^{3}-6\rho^{4}-2\rho^{5}+2\rho^{6}$ in the interval $[0,1]$, numerically, $\tilde{\rho} \approx 0.269187$, with corresponding minimal correlation coefficient $C_{\tilde{\rho}} \approx -0.444039$. The correlation coefficient $C_{\rho}$, considered as a function of the ratio $\rho$, is illustrated in Figure~\ref{fig:Corr}.

\begin{figure}[htb]
\begin{center}
\includegraphics[height=4cm]{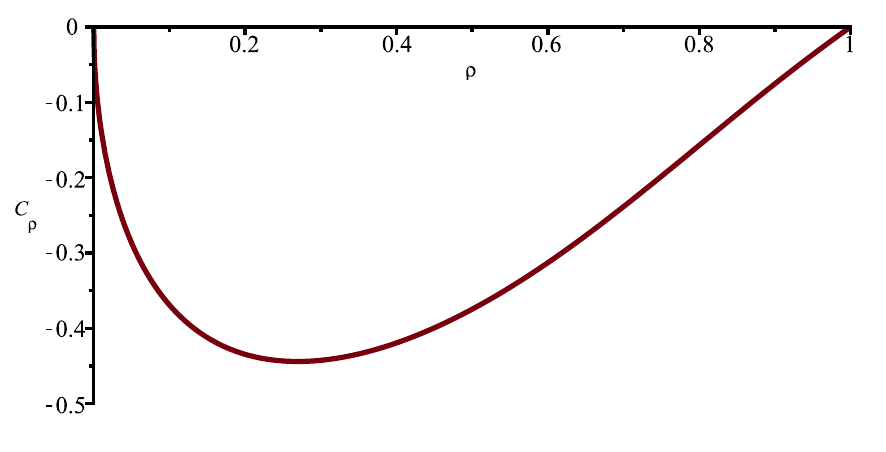}

\vspace*{-5mm}

\caption{The correlation coefficient $C_{\rho} = \Corr(\hat{X}_{\rho},\hat{Y}_{\rho})$ as a function of $\rho$ in the interval $[0,1]$.\label{fig:Corr}}
\end{center}
\end{figure} 


\subsection{Expectation and higher moments: equal number of cards}
At the end we compare our limit laws to the result of Diaconis and Graham~\cite{DiaconisGraham1981} for the expected value.
They considered the more general case of $n$ different colours, but restricted themselves there to the special case
$m_1=m_2=\dots=m_n=m$, i.e., to the r.v.\ $C_{m,m,\dots,m}$ counting the number of correct guesses when starting with a random deck of $n m$ cards containing each of the $n$ colours exactly $m$ times.
\begin{theorem}[Diaconis-Graham~\cite{DiaconisGraham1981}]
\label{the:DG}
The expected value $S_{m,n}=\E(C_{m,m,\dots,m})$ of the number of correct guesses satisfies for fixed $n\in\N$, with $n\ge 2$, and $m\to\infty$ the expansion
\[
S_{m,n}=m + \frac{\pi}2 \cdot M_n \cdot \sqrt{m} + o(\sqrt{m}),
\]
where $M_n=\E(\max\{N_1,\dots,N_n\})$ and the random variables $N_k$ are iid standard normal distributed.
\end{theorem}

In the special case $n=2$ we observe 
\[
\E(C_{m,m})=S_{m,2}=m+\frac{\sqrt{\pi}}2 \sqrt{m} + o(\sqrt{m}),
\]
using the explicit value $1/\sqrt{\pi}$ of $M_2$. Our results in Theorem~\ref{thm:W_limit} suggest that the expected value
and higher moments for $n=2$ actually stem from the Rayleigh distribution. We obtain the following result, 
which is in accordance with Theorem~\ref{the:DG}, as $\Gamma(3/2)=\sqrt{\pi}/2$, and extends it from $s=1$ to arbitrary high integer moments.

\begin{theorem}
\label{the:raw}
The raw moments of the centered random variable $\hat{C}_{m,m}=C_{m,m}-m$, 
counting the total number of correct guesses, satisfy for $s \in \N$ and $m\to\infty$ the asymptotic expansion
\[
\E(\hat{C}_{m,m}^s)\sim \frac{1}{2^s}\E(R^s) \cdot m^{\frac{s}2}
=\Gamma\big(\frac{s}2+1\big)\cdot m^{\frac{s}2},\quad \text{with} \quad R \law \Rayleigh\big(\sqrt{2}\big).
\]
\end{theorem}

In order to prove this result, we need the asymptotics of the factorial moments of $W_{m_1,m_2}$ in the special case $m_1=m_2=m$. We use the following result, relating the factorial moments to a Cauchy integral.
\begin{lem}
\label{lem:AC1}
For $s\ge 1$ the $s$th factorial moment of $W_{m_1,m_2}$ is given by
\[
\E(\fallfak{W_{m_1,m_2}}{s})= \frac{s!2^s}{\binom{m_1+m_2}{m_1}}\frac{1}{2\pi i}\oint_{\mathcal{C}} \frac{1}{(1-x)^s}\cdot (1+x)^{m_1+m_2}\frac{dx}{x^{m_2+1-s}},
\]
with $\mathcal{C}$ a positively oriented simple closed curve around the origin. 
\end{lem}
\begin{remark}
The result above and Corollary~\ref{coro} allow to study moment convergence of $W_{m_1,m_2}$ and $\hat{C}_{m_1,m_2}$ not only for $m_1=m_2=m\to\infty$, but also for the different cases considered in Theorem~\ref{thm:W_limit}. However, the analysis gets much more involved and we refrain from going into details.
\end{remark}

\begin{proof}[Proof of Lemma~\ref{lem:AC1}]
It is convenient to consider the binomial moments $\E\binom{W_{m_1,m_2}}{s}$, such that 
\[
\E(\fallfak{W_{m_1,m_2}}{s}) = s!\cdot \E\binom{W_{m_1,m_2}}{s}.
\]
From the explicit result in Corollary~\ref{cor:Marginal_dist} we get
\begin{align*}
\E\binom{W_{m_1,m_2}}{s}&=\sum_{k=s}^{m_2}\binom{k}s\P(W_{m_1,m_2}=k)\\
&=\frac{1}{\binom{m_1+m_2}{m_1}}\sum_{k=s}^{m_2}\binom{k}s\Big(2^{k} \binom{m_1+m_2-k}{m_1} - 2^{k+1} \binom{m_1+m_2-k-1}{m_1}\Big).
\end{align*}
By the standard property $\binom{k+1}s=\binom{k}s+\binom{k}{s-1}$ we obtain further
\begin{align*}
&\E\binom{W_{m_1,m_2}}{s}=\sum_{k=s}^{m_2}\binom{k}s\P(W_{m_1,m_2}=k)\\
&\,=\frac{1}{\binom{m_1+m_2}{m_1}}\sum_{k=s}^{m_2}\Big(\binom{k}s 2^{k} \binom{m_1+m_2-k}{m_1} - \binom{k+1}s2^{k+1} \binom{m_1+m_2-k-1}{m_1}\Big)\\
&\quad + \frac{1}{\binom{m_1+m_2}{m_1}}\sum_{k=s}^{m_2-1}\binom{k}{s-1}2^{k+1} \binom{m_1+m_2-k-1}{m_1}.
\end{align*}
The telescoping sum simplifies to $\frac{2^{s} \binom{m_1+m_2-s}{m_1}}{\binom{m_1+m_2}{m_1}}$. 
Further simplification gives 
\[
\E\binom{W_{m_1,m_2}}{s}= 
\frac1{\binom{m_1+m_2}{m_1}}\sum_{k=s-1}^{m_2-1}\binom{k}{s-1}2^{k+1} \binom{m_1+m_2-k-1}{m_1}\Big).
\]
Using the standard expansion for integers $n,j\ge 0$:
\[
[t^n]\frac1{(1-2t)^{j+1}}=2^n\binom{n+j}{j},
\]
we obtain 
\[
[t^{m_2-1}]\frac{t^{s-1}}{(1-2t)^{s}}\cdot \frac1{(1-t)^{m_1+1}}
=\sum_{k=s-1}^{m_2-1}\binom{k}{s-1}2^{k+1-s} \binom{m_1+m_2-k-1}{m_1}.
\]
Thus, by an application of Cauchy's integration formula we obtain the following contour integral representation of the $s$th factorial moments of $W_{m_1,m_2}$:
\[
\E(\fallfak{W_{m_1,m_2}}{s})= \frac{2^s\cdot s!}{\binom{m_1+m_2}{m_1}}\cdot
\frac{1}{2\pi i}\oint_{\mathcal{C}} \frac{1}{(1-2t)^{s}}\cdot \frac1{(1-t)^{m_1+1}}\frac{dt}{t^{m_2-s+1}}.
\]
Finally, we use the substitution $t=\frac{x}{x+1}$ to get the stated result.
\end{proof}
Next we are going to asymptotically evaluate the integral for $m_1=m_2=m\to\infty$. 
\begin{proof}[Proof of Theorem~\ref{the:raw}]
Due to the relation between factorial moments and raw moments 
\[
\E(\hat{C}_{m,m}^s)=\sum_{j=0}^{s}\Stir{s}{j}\E(\fallfak{\hat{C}_{m,m}}{j}),
\]
where $\Stir{s}{j}$ denote the Stirling numbers of the second kind, it suffices to prove the stated expansion for the factorial moments.
By Corollary~\ref{coro} and Lemma~\ref{lem:AC1} we have
\[
\E(\fallfak{\hat{C}_{m,m}}{s})=
\frac{s!}{\binom{2m}{m} 2\pi i}\oint_{\mathcal{C}} \frac{1}{(1-x)^s}\cdot (1+x)^{2m}\frac{dy}{x^{m+1-s}}.
\]
Here, we are in a situation similar to the example considered by Flajolet and Sedgewick~\cite[VIII.39, page 590]{FlaSed}: 
the saddle point method is not directly applicable due to the singularity at $x=1$. 
However, we can circumvent a more difficult analysis using a trick~\cite{FlaSed}, namely another substitution
\[
y=\frac{x}{(1+x)^2}, \quad x(y)=\frac{1-2y-\sqrt{1-4y}}{2y} \sim 1-2\sqrt{1-4y}+\mathcal{O}(1-4y),
\]
where the latter local expansion holds in a slit neighbourhood of the dominant singularity $y=1/4$. 
This leads to
\begin{align*}
\E(\fallfak{\hat{C}_{m,m}}{s})&=
\frac{s!}{\binom{2m}{m}2\pi i}\oint_{\mathcal{C}} \frac{\big(1+x(y)\big)x(y)^{s}}{(1-x(y))^{s+1}}\frac{dy}{y^{m+1}}.
= \frac{s!}{\binom{2m}{m}}[y^{m}]\frac{\big(1+x(y)\big)x(y)^{s}}{(1-x(y))^{s+1}}.
\end{align*}
We can use now standard singularity analysis~\cite{FlaSed}. An expansion around $y=1/4$ gives 
\[
\frac{\big(1+x(y)\big)x(y)^{s}}{(1-x(y))^{s+1}}\sim 
\frac{1}{2^{s}(1-4y)^{\frac{s+1}2}}.
\]
Thus, singularity analysis and the classical expansion
\[
\frac{1}{\binom{2m}m}\sim \frac{\sqrt{\pi} \sqrt{m}}{4^m}
\]
yield
\begin{align*}
\E(\fallfak{\hat{C}_{m,m}}{s})&\sim
\frac{s!m^{\frac{s}2} \sqrt{\pi}}{2^s\Gamma(\frac{s+1}2)}.
\end{align*}
By the Legendre duplication formula
\[
\frac{\sqrt{\pi} \, \Gamma(s+1)}{2^s\Gamma(\frac{s+1}2)}=\Gamma\big(\frac{s}2+1\big),
\]
this leads to the stated result
\[
\E(\fallfak{\hat{C}_{m,m}}{s})\sim \Gamma\big(\frac{s}2+1\big)m^{\frac{s}2}.
\]
\end{proof}

At the very end, we also sketch an alternative approach to Theorem~\ref{the:raw}, 
based on a local limit theorem; compare also with Zagier's analysis~\cite{Zagier1990} for the expected value.
Following our previous analysis of the joint limit laws and using
\[
\P\{W_{m,m}=k\}=\frac{2^k\binom{2m-k-1}{m-1-1}\frac{k}{m}}{\binom{2m}m},\quad 0\le k\le m,
\]
we observe that, for $k$ of order $\sqrt{m}$, 
\[
\P\{W_{m,m}=k\}\sim \frac{k}{2m}e^{-\frac{k^2}{4m}}.
\]
Consequently, by Euler's summation formula
\[
\E(W_{m,m}^s)\sim \sum_{k=0}^{\infty}k^s\frac{k}{2m}e^{-\frac{k^2}{4m}} \sim 
\int_0^{\infty}x^s\cdot \frac{x}{2m}e^{-x^2/4m}dx.
\]
As $f(x)=\frac{x}{2m}e^{-x^2/4m}$ is the density of a Rayleigh-distributed random variable with parameter $\sigma=\sqrt{2m}$, 
we obtain further
\[
\E(W_{m,m}^s)\sim \sigma^s\big(\sqrt{2}\big)^s\Gamma\big(\frac{s}2+1\big) =2^s\Gamma\big(\frac{s}2+1\big) m^{\frac{s}2}.
\]

\section*{Declarations of interest}
The authors declare that they have no competing financial or personal interests that influenced the work reported in this paper. 

\bibliographystyle{cyrbiburl}
\bibliography{CardGuessingTwo-refs}{}


\section*{Appendix: Additional combinatorial models}
In the appendix we relate the card guessing game with two types of cards to different combinatorial problems. 
This allows to directly translate the results for the random variable $W_{m_1,m_2}$, see Corollary~\ref{cor:Marginal_dist}, Theorems~\ref{thm:W_limit},~\ref{thm:WT-JointLimit} and Lemma~\ref{lem:AC1},
to the different random variables stated in the following.

\smallskip

First we recall the definition of the sampling without replacement urn. 
An urn contains two types of balls, say $m_1$ red and $m_2$ black balls, with $m_1,m_2\in\N_0$. 
The urn evolves by successive draws of random balls at discrete instance according to the transition matrix 
$\left(
\begin{smallmatrix}
-1 & 0\\
0 & -1\\
\end{smallmatrix}
\right)
$, which means that the colour of the drawn ball is inspected and then the ball is discarded. 
Usually, one is interested in the composition of the urn once a colour is fully depleted, 
but here we continue this process until the urn is completely empty. 

Let $G_{m_1,m_2}$ denote the random variable counting the number of times there are equally many red and black balls in the sampling without replacement urn process, except for the empty urn, starting with $m_1$ red and $m_2$ black balls, $m_1\ge m_2\ge 0$.
\begin{prop}[Card guessing and equality in the sampling without replacement urn]
\label{prop:appendix1}
The random variable $G_{m_1,m_2}$ has the same distribution as the random variable $W_{m_1,m_2}$, counting the number of times during the card guessing process when the number of red and black cards are equal and non-zero.
\end{prop}
\begin{remark}
Time reversal~\cite{FlaDuPuy2006,KV2003} also allows to obtain a relation of to the number of equalities in 
a standard P\'olya urn with ball transition matrix$\left(
\begin{smallmatrix}
1 & 0\\
0 & 1\\
\end{smallmatrix}
\right)
$.
\end{remark}
\begin{proof}
Similar to the geometric interpretation of the card guessing game in Subsection~\ref{sec:Geometric_interpretation}, 
one also thinks of evolution of the sampling urn in terms of lattice paths~\cite{FlaDuPuy2006,HKP2007}.
However, we note that the state spaces are different: for sampling without replacement we consider $\Omega_{S}=\N_0\times \N_0$, 
whereas for the card guessing game we consider only $\Omega_{C}=\{(x,y)\in\N_0\times \N_0\colon x\ge y\ge 0\}$.
In order to prove $\P\{G_{m_1,m_2}=k\}=\P\{W_{m_1,m_2}=k\}$, we proceed by constructing a surjection $\psi$ from 
$\Omega_{S}$ to $\Omega_{C}$. Each sample path $\omega\in\Omega_{S}$ from the sampling without replacement urn, contributing to $\P\{G_{m_1,m_2}=k\}$, is either directly contained in $\Omega_{C}$, or it has a subpath going above the diagonal $y=x$. 

\begin{figure}[!htb]
\begin{center}
\includegraphics[scale=0.45]{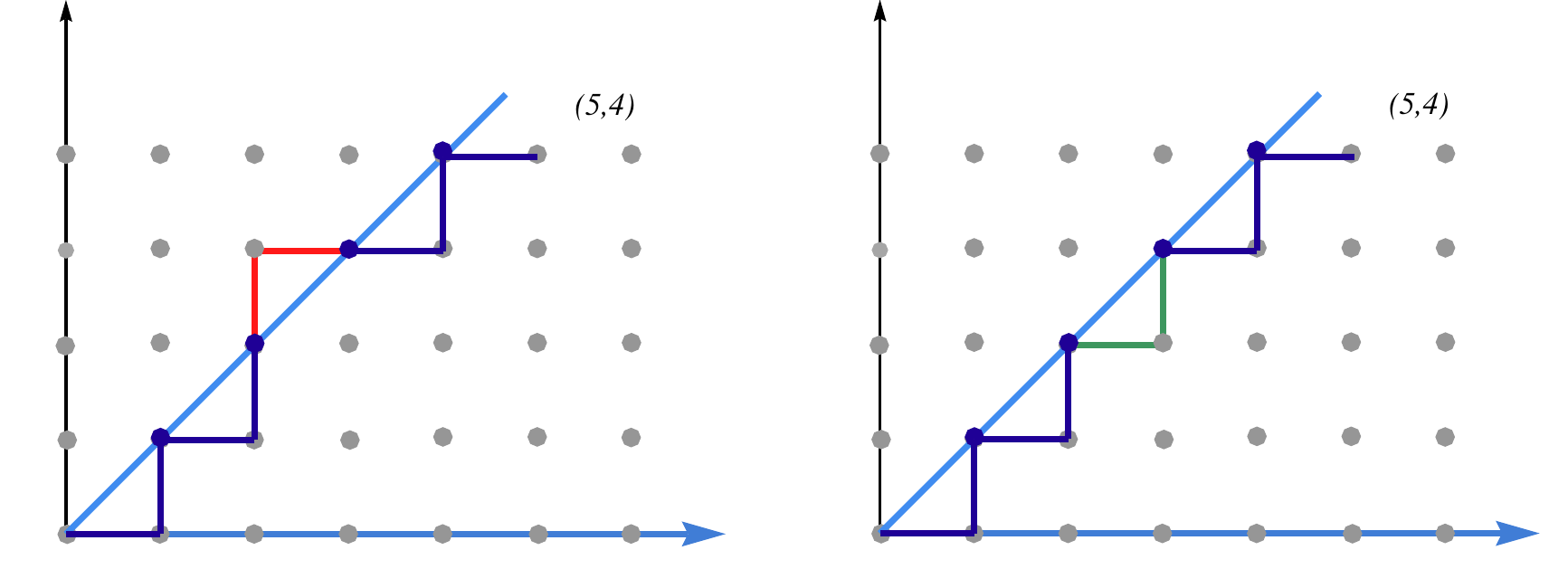}
\caption{A sample path $\sigma\in\Omega_{S}$ and its image $\psi(\sigma)=\omega\in\Omega_C$.}
\end{center}
\end{figure} 

The surjection $\psi$ maps all paths $\sigma\in\Omega_{S}\setminus\Omega_{C}$ to a path $\omega$ in $\Omega_{C}$ by mirroring the parts above the diagonal. Moreover, its restriction is the identity: $\psi |_{\Omega_{C}}=\text{id}_{\Omega_{C}}$. Comparing the weight of a path $\omega\in\Omega_{C}$ in the sampling without replacement urn and the card guessing process
with 
\[
W_{m_1,m_2}(\omega)=G_{m_1,m_2}(\omega)=k,
\]
we observe $\P_S(\omega)=\frac1{2^k}\P_C(\omega)$. All $2^k-1$ different path $\sigma\in\Omega_{S}\setminus\Omega_{C}$
with $\psi(\sigma)=\omega$ have the same probability (weight) as $\omega$, and thus 
\begin{align*}
\P\{W_{m_1,m_2}=k\}&=\sum_{\substack{\omega\in\Omega_{C}\\W_{m_1,m_2}(\omega)=k}}\P_C(\omega)
=\sum_{\substack{\omega\in\Omega_{C}\\W_{m_1,m_2}(\omega)=k}}\sum_{\sigma\in\psi^{-1}(\omega)}\P_S(\sigma)\\
&=\sum_{\substack{\sigma\in\Omega_{S}\\ G_{m_1,m_2}(\sigma)=k}}\P_S(\sigma)=\P\{G_{m_1,m_2}=k\}.
\end{align*}
\end{proof}

Next, we relate the problem to classical Dyck paths~\cite{BanderierFlajolet2002}. 
\begin{prop}[Card guessing and Dyck paths]
\label{prop:appendix2}
The random variable $W_{m_1,m_2}$ has the same distribution
as the random variable $X_{m_1+m_2}$, counting the number of returns to zero in a Dyck 
walk of length $m_1+m_2$, with final altitude $m_2-m_1$. 
In the special case $m_1=m_2=m$ we observe a Dyck bridge of length $2m$.
\end{prop}
\begin{remark}
The relation between Card guessing, Sampling without replacement and Dyck paths in the Proposition above and~\ref{prop:appendix1} also allows 
to interpret hitting times considered in~\cite{KuPanPro2009} in terms of Dyck paths.
\end{remark}

\begin{proof}
We actually show that $G_{m_1,m_2}\law X_{m_1+m_2}$, which by Proposition~\ref{prop:appendix1} leads to the stated result. 
First, we observe that for each path from $\sigma\in\Omega_{S}$ it holds that
\[
\binom{m_1+m_2}{m_1}\cdot \P\{\sigma\}=1. 
\]
Thus, instead of considering the step by step evolution of the urn we can enumerate all paths
from $(m_1,m_2)$ to $(0,0)$, touching (or crossing) the diagonal exactly $k$ times. Then, 
the probability is determined all such paths divided by their total number $\binom{m_1+m_2}{m_1}$. 

\begin{figure}[!htb]
\begin{center}
\includegraphics[scale=0.45]{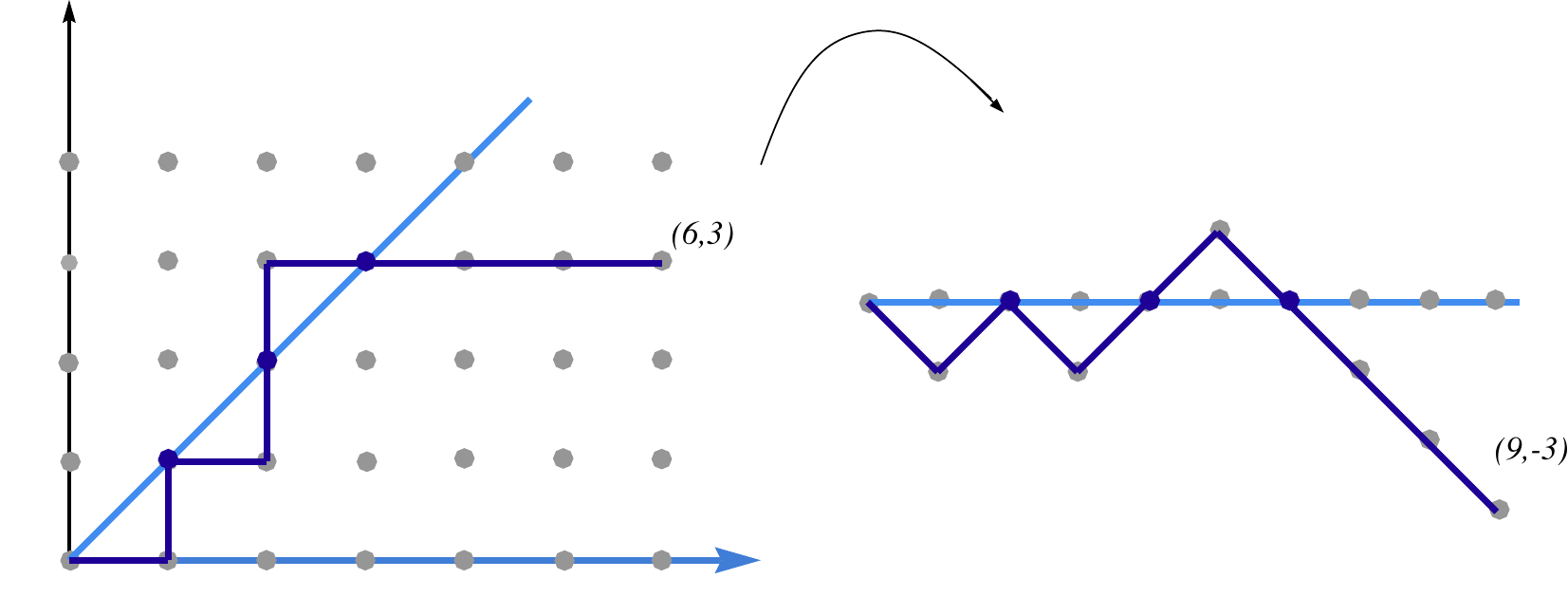}
\caption{A sample path $\omega\in\Omega_{S}$ starting at $(6,3)$ and the corresponding Dyck walk of length $9$ with final altitude $-3$.}
\end{center}
\end{figure} 
After rotation of the coordinate system, we reverse the direction of paths. 
The new steps are $(1,1)$ and $(1,-1)$ and the length of the walks is given by the sum $m_1+m_2$. Thus, we observe the relation to the stated classical Dyck walks, where the final altitude is given by 
$m_2-m_1$. 

\begin{figure}[!htb]
\begin{center}
\includegraphics[scale=0.45]{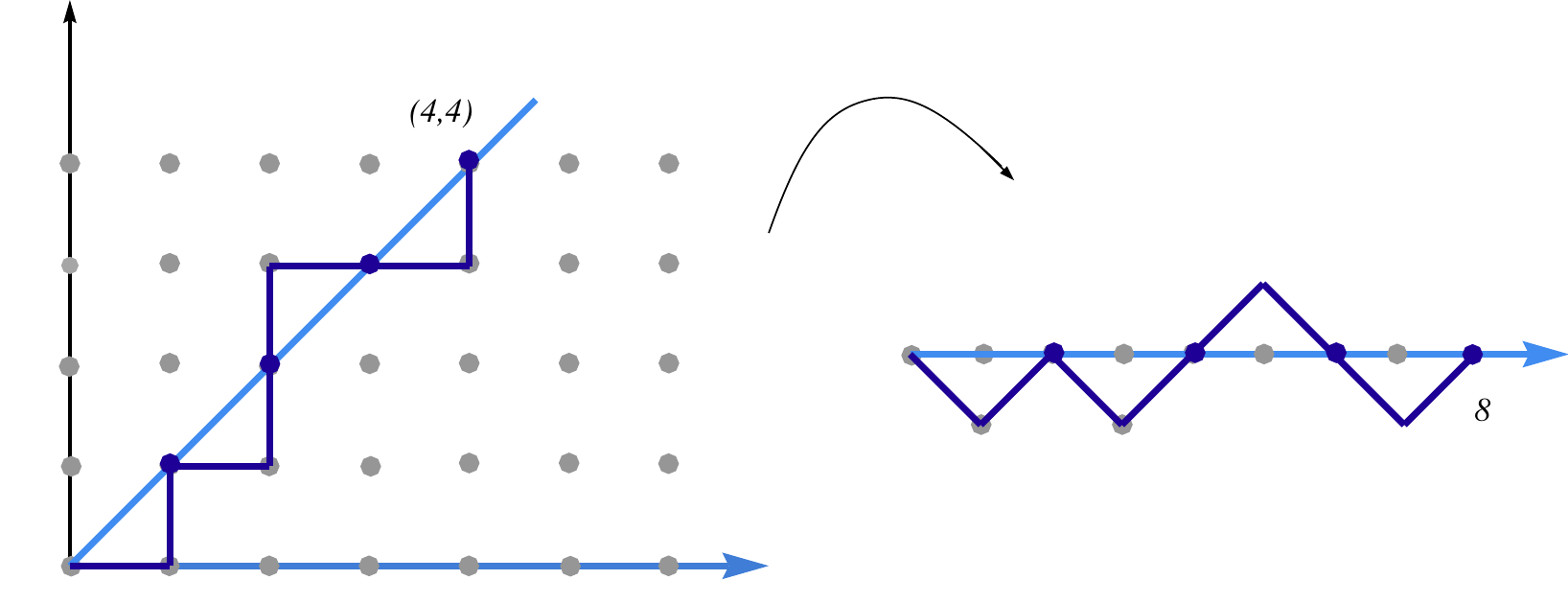}
\caption{A sample path $\omega\in\Omega_{S}$ starting at $(4,4)$ and the corresponding Dyck bridge of length $8$.}
\end{center}
\end{figure} 

In particular, for $m_2-m_1=0$ one obtains ordinary Dyck bridges of length $2m$. This can be used to give in the special case $m_1=m_2=m$ a complete different proof of the Rayleigh limit law
for both $W_{m,m}$, as well as $\hat{C}_{m,m}$.
\end{proof}

\end{document}